\documentclass{amsart}
\usepackage{latexsym,amsxtra,amscd,ifthen}
\usepackage{amsfonts}
\usepackage{verbatim}
\usepackage{amsmath}
\usepackage{amsthm}
\usepackage{amssymb}

\input xy \xyoption{matrix}
\xyoption{arrow}\xyoption{frame}

\newcommand{\edge}{\ar@{-}}

\numberwithin{equation}{section}

\theoremstyle{plain}
\newtheorem{theorem}{Theorem}[section]
\newtheorem{lemma}[theorem]{Lemma}
\newtheorem{proposition}[theorem]{Proposition}
\newtheorem{corollary}[theorem]{Corollary}
\newtheorem{conjecture}[theorem]{Conjecture}

\theoremstyle{definition}
\newtheorem{definition}[theorem]{Definition}

\newtheorem{remark}[theorem]{Remark}
\newtheorem{remarks}[theorem]{Remarks}
\newtheorem*{remark*}{Remark}
\newtheorem{question}[theorem]{Question}

\newcommand{\gr}{\operatorname{gr}}
\newcommand{\fm}{\mathfrak m}
\newcommand{\fb}{\mathfrak b}
\newcommand{\fg}{\mathfrak g}

\newcommand{\GK}{\operatorname{GKdim}}

\newcommand{\gnoc}{\mathrel{{\lower.2ex\hbox{$\backsim$}}\llap{\raise.45ex\hbox{=}}}}

\DeclareMathOperator\GKdim{\operatorname{GKdim}}
\DeclareMathOperator\Ext{\operatorname{Ext}}

\DeclareMathOperator\gldim{\operatorname{gldim}}

\begin{document}

\title{Connected (graded) Hopf algebras}

\author{K.A. Brown}
\author{P. Gilmartin}
\address{School of Mathematics and Statistics\\
University of Glasgow\\ Glasgow G12 8QW\\
Scotland}
\email{ken.brown@glasgow.ac.uk}
\email{p.gilmartin.1@research.gla.ac.uk}

\author{J.J. Zhang}
\address{Department of Mathematics\\
Box 354350\\
University of Washington\\
Seattle, WA 98195, USA}
\email{zhang@math.washington.edu}


\begin{abstract}
We study algebraic and homological properties of
two classes of infinite dimensional Hopf algebras
over an algebraically closed field $k$ of characteristic
zero. The first class consists of those Hopf $k$-algebras
that are connected graded as algebras, and the second
class are those Hopf $k$-algebras that are connected
as coalgebras. For many but not all of the results
presented here, the Hopf algebras are assumed to have
finite Gel'fand-Kirillov dimension. It is shown that
if the Hopf algebra $H$ is a connected graded Hopf
algebra of finite Gel'fand-Kirillov dimension $n$, then
$H$ is a noetherian domain which is Cohen-Macaulay,
Artin-Schelter regular and Auslander regular of global
dimension $n$. It has $S^2 = \mathrm{Id}_H$, and is
Calabi-Yau. Detailed information is also provided
about the Hilbert series of $H$. Our results leave
open the possibility that the first class of algebras
is (properly) contained in the second. For this second
class, the Hopf $k$-algebras of finite Gel'fand-Kirillov
dimension $n$ with connected coalgebra, the underlying
coalgebra is shown to be Artin-Schelter regular of global
dimension $n$. Both these classes of Hopf algebra share
many features in common with enveloping algebras of finite
dimensional Lie algebras. For example, an algebra in
either of these classes satisfies a polynomial identity
only if it is a commutative polynomial algebra.
Nevertheless, we construct, as one of our main
results, an example of a Hopf $k$-algebra $H$ of
Gel'fand-Kirillov dimension 5, which is connected graded as
an algebra and connected as a coalgebra, but is not 
isomorphic as an algebra
to $U(\mathfrak{g})$ for any Lie algebra $\mathfrak{g}$.
\end{abstract}

\maketitle

\section*{Introduction}
\label{yysec0}
Throughout $k$ is an algebraically closed field of
characteristic $0$. Unless otherwise stated, all vector
spaces and tensor products are over $k$.

\subsection{Connectedness}
\label{yysec0.1}
The adjective ``connected'' can of course be applied to a multitude
of mathematical structures, with a corresponding multiplicity of
meanings. We are concerned here with two such structures,
connected graded associative $k$-algebras, and connected 
$k$-coalgebras.  On the
one hand, a \emph{connected graded} $k$-algebra $A$ is one which
is $\mathbb{N}$-graded, with $A = k \bigoplus_{i \geq 1}A(i)$; on the
other hand, a \emph{connected coalgebra} $C$ is a coalgebra whose
coradical $C_0$ is one-dimensional
\footnote{The term ``connected coalgebra'' was used by Milnor and
Milnor-Moore in their seminal work \cite{Mil, MM} to refer to an
$\mathbb{N}$-\emph{graded} coalgebra $C$ with $C_0 = k$, thus
matching the usage in the algebra setting. But the standard modern
usage for Hopf algebras is the one described here, that is, with
no graded hypothesis on the coalgebra.}.
It is in this second sense that the adjective is traditionally
applied to a Hopf algebra $H$: that is, the Hopf algebra $H$
(whether graded or not) is said to be \emph{connected} if its
coradical $H_0$ is $k$, \cite[Definition 5.1.5]{Mo}. One main
purpose of this paper is to study how these two versions of
connectedness interact in the setting of Hopf $k$-algebras,
where one, both or neither of them may hold for a given such
algebra.

Connected Hopf $k$-algebras $H$ of finite Gel'fand-Kirillov
dimension (denoted by $\GKdim H < \infty$) have been the subject
of a number of recent papers, for example \cite{EG, Zh, Wa1, BO},
and in part this paper continues that programme. We shall also
consider the question: If $H$ is a Hopf algebra with finite
 Gel'fand-Kirillov dimension,
what is the effect on its structure of assuming that $H$
\emph{is connected graded as an algebra}? As will become clear,
these two hypotheses - connected as a graded algebra, and connected
as a coalgebra - are intimately linked.

\subsection{The cocommutative and commutative cases}
\label{yysec0.2}
The \emph{cocommutative} connected Hopf $k$-algebras were
completely described half a century ago: by the celebrated theorem
of Cartier-Kostant \cite[Theorem 5.6.5]{Mo}, such Hopf algebras
$H$ are precisely the universal enveloping algebras $U(\mathfrak{g})$
of the $k$-Lie algebras $\mathfrak{g}$, with $\mathfrak{g}$
constituting the space of primitive elements of $H$. Such an
enveloping algebra $U(\mathfrak{g})$ has $\GKdim H = n < \infty$
if and only if $\dim_k \mathfrak{g} = n$. Suppose on the other hand
that a cocommutative Hopf $k$-algebra $H$ is connected graded as an
algebra. These conditions trivially ensure that the only group-like
element of $H$ is the identity element. Thus, since $H$ is pointed
as the ground field $k$ is algebraically closed \cite[$\S$5.6]{Mo},
it is connected as a coalgebra and the Cartier-Kostant theorem
again applies, yielding $H \cong U(\mathfrak{g})$ once more.
Moreover, this time the connected graded hypothesis implies that,
if $\GKdim H < \infty$, then $\mathfrak{g}$ is nilpotent.

There is an equally beautiful story in the other ``classical''
setting, where $H$ is \emph{commutative}. For a commutative
Hopf $k$-algebra $H$ the two finiteness conditions - being affine
and being noetherian - coincide by Molnar's theorem \cite{Mol},
and this hypothesis implies finiteness of Gel'fand-Kirillov
dimension. Writing $\dim U$ for the dimension of the variety $U$
and $\gldim H$ for the global dimension of $H$, we have:

\begin{theorem}
\label{yythm0.1}
Let $H$ be an affine commutative Hopf $k$-algebra. Then the
following are equivalent:
\begin{enumerate}
\item[(1)]
$H$ is a connected Hopf algebra.
\item[(2)]
$H$ is a connected graded algebra.
\item[(3)]
For some $n \geq 0$, $H \cong k[x_1, \ldots , x_n]$ as an algebra.
\item[(4)]
There is a unipotent $k$-group $U$ such that $H$ is the
algebra of regular functions $\mathcal{O}(U)$.
\end{enumerate}
When the above hold,
$$ n = \GKdim H = \dim U = \gldim H.$$
\end{theorem}

Here $(2)\Rightarrow (3)$ follows since $H$ has finite global
dimension by \cite[combining \S11.4 and \S11.6]{Wat}, so a
theorem of Serre applies \cite[Theorem 6.2.2]{Be};
$(3) \Rightarrow (4)$ is a theorem of Lazard \cite{Laz},
$(3)\Rightarrow(2)$ is trivial, and the remaining implications
$(1) \Leftrightarrow (4)$ and $(4) \Rightarrow (3)$ are basic
facts concerning algebraic $k$-groups, see for example
\cite[Example 3.1]{BO}.
\footnote{Connected graded Hopf algebras have been studied
by Milnor-Moore in \cite{MM} and F{\'e}lix-Halperin-Thomas in
\cite{FHT1, FHT2} (assuming that both the coalgebra \emph{and}
the algebra structure are graded). Dualizing complexes
and homological properties of locally finite ${\mathbb N}$-graded Hopf
algebras have been studied in \cite{WZ1}.}

Summarising this classical picture: if $H$ is a cocommutative
or commutative Hopf algebra with finite Gel'fand-Kirillov
dimension $n$ and is connected graded as an algebra,
then it is a connected coalgebra, (and hence, as an algebra,
it is the enveloping algebra of an $n$-dimensional nilpotent
Lie algebra); and the converse (connected coalgebra implies
connected graded as an algebra) holds in the commutative case
but not the cocommutative one.

\subsection{Connected graded algebras}
\label{yysec0.3}
The first part of the paper, $\S\S$ \ref{yysec1} and
\ref{yysec2}, concerns a Hopf algebra that is connected
graded as an algebra. Our first main result gives structural
information for connected graded Hopf algebras of finite
GK-dimension. It shows that, while we cannot replicate
the cocommutative conclusions, many of the algebraic and
homological consequences of being isomorphic to the
enveloping algebra of a nilpotent Lie algebra do survive:

\begin{theorem}
\label{yythm0.2}
Let $H$ be a Hopf $k$-algebra with $\mathrm{GKdim}H = n
< \infty$ which is connected graded as an algebra. Then
$H$ is a noetherian domain with $S^2 = \mathrm{Id}_H$.
Furthermore, it is Cohen-Macaulay, Artin-Schelter regular and
Auslander regular of global dimension $n$, and is
Calabi-Yau. As a consequence, $H$ satisfies Poincar{\'e} duality.
\end{theorem}

Theorem \ref{yythm0.2} is Theorem \ref{yythm2.2} below. There 
is a result similar to Theorem \ref{yythm0.2} for connected 
Hopf algebras, see \cite[Theorem 6.9 and Corollary 6.10]{Zh}.
In \cite[Question C]{Br1} the first-named author
asks: {\it Is there an easy method to
recognize when a noetherian Hopf algebra is regular, namely,
has finite global dimension?} Theorem \ref{yythm0.2}
answers \cite[Question C]{Br1} (as well as \cite[Question E]{Br2}
and \cite[Question 3.5]{Go2}) in the case of connected graded algebras.
Faced with the above list of ring-theoretic and homological
properties and considering the cocommutative and commutative
cases discussed in $\S$\ref{yysec0.2}, one might speculate
that every Hopf algebra $H$ as in Theorem \ref{yythm0.2}
is in fact isomorphic \emph{as an algebra} to the enveloping
algebra of a nilpotent Lie algebra. But this is not so -
see Theorem \ref{yythm0.5}(2) below.

Notwithstanding this example, under some additional
hypotheses only enveloping algebras can occur in the
circumstances of Theorem \ref{yythm0.2}. Two such cases
appear as parts (2) and (4) of the following result, which
is also Theorem \ref{yythm2.4}. The graded Hopf algebras
studied by F{\'e}lix-Halperin-Thomas
in \cite{FHT1, FHT2} are in the super sense, which means that
one needs to use the Koszul sign rule when commuting two
homogeneous symbols. We expect that there is a version of
Theorem \ref{yythm0.2} for such Hopf algebras.

\begin{theorem}
\label{yythm0.3}
Let $H$ be a Hopf $k$-algebra which is connected graded and locally
finite as an algebra.
\begin{enumerate}
\item[(1)]
The Hilbert series of $H$ is
$$\frac{1}{\prod_{i=1}^\infty (1-t^i)^{n_i}}$$
for some non-negative integers $n_i$. Further, the sequence
$\{n_i\}_{i\geq 1}$ is uniquely determined by the Hilbert
series of $H$.
\item[(2)]
Suppose $H$ is generated in degree 1. Then $H$ is
isomorphic as an algebra to $U({\mathfrak g})$ for a graded Lie algebra
${\mathfrak g}$ generated in degree 1. Hence, for all $i \geq 1$,
$$\dim_k \mathfrak{g}_i = n_i,$$
where $\mathfrak{g}_i$ is the $i^{\mathit{th}}$ degree component
of ${\mathfrak g}$.
\item[(3)] Suppose again that $H$ is generated in degree 1, $H \neq k$. Then
$$ \{j: n_j \neq 0 \}  = [1,\ell) \cap \mathbb{N},
\textit{ for some } \ell \in [2, \infty]. $$
\item[(4)]
If the Hilbert series of $H$ is $\frac{1}{(1-t)^d}$, then $H$
is commutative, and hence $H \cong k[x_1, \dots , x_d]$ as algebras.
\item[(5)]
$\GKdim H < \infty$ if and only if $\sum_{i=1}^{\infty} n_i<\infty$.
In this case, $$\GKdim H=\sum_{i=1}^{\infty} n_i.$$
\end{enumerate}
\end{theorem}

Both Theorems \ref{yythm0.2} and \ref{yythm0.3} can be used
to show that a number of large classes of connected graded
$k$-algebras \emph{do not} admit any Hopf algebra structure
- see Corollaries \ref{yycor2.3} and \ref{yycor2.8}.

\subsection{Connected Hopf algebras}
\label{yysec0.4}
A central underlying theme of this paper is the following.
Let $H$ be a Hopf $k$-algebra, say of finite
Gel'fand-Kirillov dimension. What is the relationship
between the following two hypotheses on $H$?
\begin{align*}
\mathrm{(CGA)}:
& \quad \textit{Hopf algebra $H$ is connected graded as an algebra.}\\
\mathrm{(CCA)}:
& \quad \textit{Hopf algebra $H$ is connected as a coalgebra.}
\end{align*}
It is immediately clear that (CCA)$\nRightarrow$(CGA) - consider
any enveloping algebra $H$ of a finite dimensional non-nilpotent
Lie algebra, with its standard cocommutative coalgebra structure.
It remains feasible, however, that (CGA)$\Rightarrow$(CCA); indeed,
all the results discussed in $\S\S$\ref{yysec0.2} and \ref{yysec0.3}
are consistent with this possibility, which we therefore leave
as an open question:

\begin{question}
\label{yyque0.4} Does $\mathrm{(CGA)}\Rightarrow\mathrm{(CCA)}$?
\end{question}

Note that if $H$ satisfies $\mathrm{(CGA)}$, then its graded
dual is a Hopf algebra satisfying $\mathrm{(CCA)}$.

A further possibility suggested by the results of
$\S\S$\ref{yysec0.2} and \ref{yysec0.3}, already mentioned in
$\S$\ref{yysec0.3}, is that \emph{if $H$ satisfies} (CGA)
\emph{and/or} (CCA) \emph{then} $H$ \emph{is isomorphic as an
algebra to the enveloping algebra of a finite dimensional Lie
algebra.} However, the second part of the following theorem 
rules this out. It is
proved as Theorem \ref{yythm5.6},  the main result of $\S$5.
Recall that a \emph{graded Hopf algebra} $H$ is by definition
simultaneously graded as both an algebra and as a coalgebra:
$H = \oplus_{i \geq 0} H(i)$ as a graded algebra, with
$\Delta (x) \in \sum_{i=1}^m H(i)
\otimes H(m-i)$ for all $x \in H(m)$, for all $m \geq 0$ and
the antipode being a graded map
\cite[Definition 1, p.878]{Zh}. Such a graded Hopf algebra is then
\emph{connected}, in \emph{both} the coalgebra and the algebra
sense of the word, if $H(0) = k$.

\begin{theorem}
\label{yythm0.5}
Let $H$ be a Hopf $k$-algebra with $\GKdim H = n < \infty$.
\begin{enumerate}
\item[(1)]
If $H$ satisfies $\mathrm{(CCA)}$ and $n \leq 4$ then $H$
is isomorphic as an algebra to the enveloping algebra of
an $n$-dimensional Lie algebra.
\item[(2)]
There exists an algebra $H$ with $n=5$, satisfying both
$\mathrm{(CCA)}$ and $\mathrm{(CGA)}$, which is not
isomorphic to the enveloping algebra of a Lie algebra.
\item[(3)] The algebra $H$ of part {\rm{(2)}} is a
connected graded Hopf algebra.
\end{enumerate}
\end{theorem}

Theorem \ref{yythm0.5}(1) follows easily from the
classification in \cite{Wa1} of all $H$ with $\GKdim H
\leq 4$ satisfying $\mathrm{(CCA)}$. The analogue of
Theorem \ref{yythm0.5}(1) for the $\mathrm{(CGA)}$
hypothesis remains open; obviously it would follow from a
positive answer to Question \ref{yyque0.4} for $n \leq 4$.
Note that Theorem \ref{yythm0.5}(2) gives a negative answer to
\cite[Question L]{BG1}. Theorem \ref{yythm0.5}(2),(3) appear 
as Theorem \ref{yythm5.6} (and Lemma \ref{yylem5.5}) below.

Let $\gr_c H$ denote the associated
graded algebra of $H$ with respect to the coradical filtration
of $H$. The \emph{signature} featuring in Corollary
\ref{yycor0.6}(4) below is a list of the
degrees of the homogeneous generators of $\gr_c H$.
It is defined and discussed in \cite[Definition 5.3]{BG2}.

\begin{corollary}
\label{yycor0.6}
Let $H$ be as in Theorem {\rm{\ref{yythm0.5}(2),(3)}}.
Then the following hold.
\begin{enumerate}
\item[(1)]
There is a connected locally finite
${\mathbb N}$-filtration of $H$, namely, the
coradical filtration $c$, such that the associated graded ring
$\gr_c H$ is isomorphic to $k[x_1,\cdots,x_5]$ where $\deg x_i\geq 1$
{\rm{(}}but not all equal 1{\rm{)}}.
\item[(2)]
There does not exist a connected locally finite ${\mathbb N}$-filtration
${\mathcal F}$ such that
$\gr_{\mathcal F} H\cong k[x_1,\cdots,x_5]$ with $\deg x_i=1$ for all $i$.
\item[(3)]
There is a negative ${\mathbb N}$-filtration ${\mathcal F}$
of $H$ such that the associated graded ring
$\gr_{\mathcal F} H\cong U({\mathfrak g})$ for some
5-dimensional graded Lie algebra ${\mathfrak g}$.
However $H$ itself is not isomorphic to $U({\mathfrak g})$
for any Lie algebra ${\mathfrak g}$.
\item[(4)]
The signature of $H$ is $(1,1,1,2,2)$.
\item[(5)]
Let $D$ be the graded $k$-linear dual of the Hopf algebra
$H$ with respect to the graded Hopf algebra structure of
Theorem {\rm{\ref{yythm0.5}(3)}}. Thus $D$ is a connected
graded Hopf algebra, affine noetherian of Gel'fand-Kirillov
dimension 5. However, $D$ is not isomorphic,
as a coalgebra, to $\gr_c K$ for any connected Hopf algebra
$K$. In particular, $D$ is not isomorphic to $\gr_c D$
as coalgebras.
\end{enumerate}
\end{corollary}

Corollary \ref{yycor0.6} is proved in $\S$5.5. Combining 
Theorem \ref{yythm0.5}(2) with Corollary \ref{yycor0.6}(4),
one sees that a connected Hopf algebra with signature $(1,1,1,2,2)$
need not be isomorphic to an enveloping algebra of an
anti-cocommutative coassociative Lie algebra. This answers negatively
a question posed rather imprecisely in \cite[Remark 2.10(d1)]{Wa1}.

Despite the example from Theorem \ref{yythm0.5}(2), the evidence
so far assembled, in $\S$$\S$\ref{yysec1} and \ref{yysec2} of the
present paper for $\mathrm{(CGA)}$ and in previous work
\cite{EG, Zh, Wa1, BO, BG2} for $\mathrm{(CCA)}$, supports the
intuition that Hopf algebras of finite Gel'fand-Kirillov dimension
satisfying either of these hypotheses share many properties
with enveloping algebras of finite dimensional $k$-Lie algebras
$\mathfrak{g}$, (with $\mathfrak{g}$ graded and hence nilpotent
in the case of $\mathrm{(CGA)}$). We provide some further
evidence in support of this philosophy, by generalising in
Theorem \ref{yythm4.5} a well-known result on enveloping
algebras of Latysev and Passman, \cite{Lat, Pas}, and also
providing a version of it for connected graded algebras,
Theorem \ref{yythm2.9}:

\begin{theorem}
\label{yythm0.7}
Let $H$ be a Hopf $k$-algebra {\rm{(}}but don't assume that
$\GKdim H$ is finite{\rm{)}}. If $\mathrm{(CGA)}$ or
$\mathrm{(CCA)}$ holds for $H$, and $H$ satisfies a polynomial
identity, then $H$ is commutative, {\rm{(}}so Theorem
{\rm{\ref{yythm0.1}}} applies{\rm{)}}.
\end{theorem}

\subsection{Further results}
\label{yysec0.5}
It turns out that the arguments employed to prove Theorem
\ref{yythm0.7} can be used, without any connectedness
hypothesis, to give insight into another type of weak
commutativity hypothesis on a Hopf algebra. We take a
brief detour to follow this line, proving, as Proposition \ref{yypro3.1}:

\begin{theorem}
\label{yythm0.8}
Let $H$ be a Hopf $k$-algebra with $\GKdim H = n < \infty$,
and let $\mathfrak{m}$ be the kernel of the counit of $H$.
\begin{enumerate}
\item[(1)]
$\dim \fm/\fm^2 \leq \GKdim H$.
\item[(2)]
Suppose that $H$ is affine or noetherian.
Then the following statements are equivalent:
\begin{enumerate}
\item[(a)]
$\dim \fm/\fm^2 = \GKdim H$.
\item[(b)]
$\mathfrak{m}$ contains a unique minimal prime ideal, $P$, 
and $H/P \cong \mathcal{O}(G)$ for a {\rm{(}}connected{\rm{)}}
algebraic group $G$ of dimension $n$.
\item[(c)]
$H$ has a commutative factor algebra $A$ with $\GKdim A = n$.
\end{enumerate}
\end{enumerate}
\end{theorem}

In another brief digression we follow up some work on Artin-Schelter
regular coalgebras in \cite{HT}. Unexplained terminology in the following
theorem is defined in $\S$\ref{yysec4.1}.

\begin{theorem}
\label{yythm0.9}
Let $H$ be a connected Hopf algebra of GK-dimension $n < \infty$.
Then the underlying coalgebra of $H$ is Artin-Schelter regular of global
dimension $n$.
\end{theorem}

\subsection{Theorem of Etingof and Gelaki}
\label{yysec0.6}
Connected Hopf algebras are studied under the name {\it coconnected
Hopf algebras} in \cite{EG}. In particular, \cite[Theorem 4.2]{EG} 
states that every connected Hopf $\mathbb{C}$-algebra is a deformation 
of a prounipotent proalgebraic groups
with a Poisson structure, and that this data classifies these objects. 
This leaves, of course, the problems of understanding the 
ring-theoretic and homological properties of these Hopf algebras, 
and of investigating the range of possibilities generated by this 
recipe. This paper is independent of the work in \cite{EG}, and it 
should be expected that further progress will follow from combining 
their result with the methods discussed here.

\subsection{}
\label{yysec0.7}
The paper is organized as follows. $\S$1 contains some preliminaries.
Hopf algebras that are connected graded as an algebra are studied in
$\S$2 and Theorems \ref{yythm0.2} and \ref{yythm0.3} are proved there.
Hopf algebras that are connected as a coalgebra are studied in $\S$4
and Theorem \ref{yythm0.9} is proved there. The (CGA)
[respectively, (CCA)] parts of Theorem \ref{yythm0.7} are proved
in $\S$\ref{yysec2.6} [resp. $\S$\ref{yysec4.2}]. The discussion
of the almost commutative Hopf algebras of Theorem \ref{yythm0.8}
is in $\S$3. $\S$5 is devoted to the proofs
of Theorem \ref{yythm0.5}(2) and Corollary \ref{yycor0.6}. Some
computational details of the proof of
Theorem \ref{yythm5.6} are relegated to the Appendix, $\S$6.

\subsection*{Acknowledgments}
Some of the research described in this paper will form part of
the second-named author's PhD thesis at the University of Glasgow,
funded by the Carnegie Trust for the Universities of Scotland.
The third-named author
was supported by the US National Science Foundation
(Nos. DMS--0855743 and DMS--1402863). Some work for this
paper was carried out at the BIRS Workshop on Nichols Algebras
and Their Interactions with Lie Theory, Hopf Algebras and
Tensor Categories, held in Banff, Canada, 6-10 September 2015.
The first and third-named authors thank BIRS for its very
supportive hospitality. We also thank Pavel Etingof and
Milen Yakimov for very helpful conversations at this meeting.

\bigskip

\section{Preliminaries}
\label{yysec1}

\subsection{General setup}
\label{yysec1.1}
Throughout let $k$ denote a base field, which we assume to be
algebraically closed and of characteristic $0$. (The algebraically closed
assumption can probably be dropped for most of our results, but
the restriction to characteristic 0 is undoubtedly needed for most
of them.) All algebraic
structures and unadorned tensor products $\otimes$ will be assumed to be over
the base field $k$. We omit the definitions of some standard homological
terms; these can be found in, for example, \cite{BZ, Le, LP, RRZ, WZ2}.
The Gel'fand-Kirillov dimension of a module $M$ over a $k$-algebra $A$ is
denoted $\GKdim M$; we use \cite{KL} as the standard reference for this.

\subsection{Hopf algebra notation}
\label{yysec1.2}
For a general Hopf algebra $H$, the usual notation $\Delta$, $\epsilon$
and $S$ is used to denote the coproduct, counit and antipode respectively.
Our standard reference for Hopf algebra material is \cite{Mo}; details of
the following, for instance, can be found in \cite[Chapter 5]{Mo}.
The \emph{coradical} filtration of $H$ is denoted by
$\{H_{i}\}_{i=0}^{\infty}$, where $H_{0}$ is the coradical of $H$, that is
the sum of all simple subcoalgebras of $H$, and we define inductively,
for all $i \geq 1$,
$$H_{i}:=\Delta^{-1}(H\otimes H_{i-1}+ H_{0}\otimes H).$$
As explained already in $\S$\ref{yysec0.1}, the Hopf algebra $H$ is
$\emph{connected}$ if $H_{0}=k$.

\subsection{Filtrations in the setting of graded algebras}
\label{yysec1.3}
We start with some material which is a slight variation of parts
of \cite{LvO1, LvO2}. Let $A=A(0)\oplus A(1)\oplus A(2)\oplus \cdots$
be an ${\mathbb N}$-graded $k$-algebra, which we shall assume unless
otherwise stated to be {\it connected} - that is, $A(0) = k$. We shall
also frequently assume that the grading of $A$ is {\it locally finite}
- that is, $\dim_k A(i)<\infty$ for all $i$. We shall call this given
grading of $A$ its {\it Adams grading}, with the degree of a homogeneous
element $w$ of $A$ denoted by $\deg_a w$. Then the {\it Hilbert series}
of $A$ is defined and denoted by
$$H_A(t)=\sum_{i=0}^{\infty} \dim_k A(i) \; t^i.$$
We shall always denote the graded maximal ideal $\oplus_{i\geq 1}A(i)$
by $\fm$, and write $d(A)$ for the minimal number of generators of $A$ as
an algebra, $d(A) \in \mathbb{N} \cup \{\infty\}$. The symbol
$A_{\mathrm{ab}}$ is used for the factor of $A$ by the ideal
$\langle [A,A] \rangle$. Note that $A_{\mathrm{ab}}$ inherits a grading
from $A$. We recall the following elementary facts.

\begin{lemma}
\label{yylem1.1}
Let $A = \oplus_{i \geq 0}A(i)$ be a connected graded locally finite
algebra, with $\fm = \sum_{i \geq 1}A(i)$.
\begin{enumerate}
\item[(1)]
$A$ is an affine algebra $\Longleftrightarrow$
$\dim_k \fm/\fm^2 =:d < \infty$. In this case, $d(A) = d$.
\item[(2)]
$d(A) = d(A_{\mathrm{ab}}).$
\end{enumerate}
\end{lemma}

The definition of a {\it Zariskian filtration} is omitted; it can be
found in many places, for example \cite{LvO1, LvO2}. We consider a
special type of ``graded Zariskian'' filtration of a graded algebra
$A$ as specified in Lemma \ref{yylem1.1}.  A family of graded vector subspaces
$\mathcal{F}:=\{F_i A\subseteq A\mid i\in {\mathbb Z}\}$ of $A$ is called a
{\it graded filtration} of $A$ if the following hold:
\begin{enumerate}
\item[(F0)]
$1\in F_0 A$ and $F_{-1}A \subseteq \fm$;
\item[(F1)]
$F_i A\subseteq F_{i+1} A$ for all $i$;
\item[(F2)]
$\bigcap_{i} F_i A=\{0\}$ and $\bigcup_{i} F_i A=A$;
\item[(F3)]
$F_i A F_j A\subseteq F_{i+j} A$ for all $i,j$.
\end{enumerate}
Define the graded ring associated to $\mathcal{F}$ to be
$$\gr_{\mathcal{F}} A:=\bigoplus_{i=-\infty}^{\infty} F_{i}A/F_{i-1} A$$
and the Rees algebra associated to $\mathcal{F}$ to be
$$\mathrm{Rees}_{\mathcal{F}} A:=\bigoplus_{i=-\infty}^{\infty} r^i F_i A,$$
where $r$ is an indeterminate. If $x\in F_{n}A\setminus F_{n-1}A$, then
the associated element in $(\gr_{\mathcal F} A)_{n}$ is denoted by $\overline{x}$.
Note that both $\gr_{\mathcal{F}} A$ and
$\mathrm{Rees}_{\mathcal{F}} A$ are ${\mathbb Z}^2$-graded
where the first component is reserved for the Adams grading and the second
component is for the grading coming from the filtration, denoted by
$\deg_{\mathcal F}$. We define $\deg_a r:=0$ and $\deg_{\mathcal F} r:=1$, or
$\deg r:=(0,1)$. Now $\gr_{\mathcal{F}} A$ and $A$ have the same Hilbert
series when we consider the grading on $\gr_{\mathcal{F}} A$ induced by
the Adams grading of $A$. In particular, $\gr_{\mathcal{F}} A$ is
connected graded with respect to the Adams grading.

The {\it standard graded filtration} of $A$ is defined by
\begin{equation}
\label{E1.1.1}\tag{E1.1.1}
F_i A=\begin{cases} A & i\geq 0\\ \fm^{-i} & i<0.\end{cases}
\end{equation}
In this case the associated graded ring $\gr_{\mathcal{F}} A$ is isomorphic to
$\bigoplus_{i=0}^{\infty} \fm^i/\fm^{i+1}$. We record for future use the
following obvious lemma:

\begin{lemma}
\label{yylem1.2}
Let $A$ be a connected graded algebra generated in degree 1. Let
$\fm= \oplus_{i\geq 1}A(i)$ and let $\mathcal{F}$ be the standard filtration.
Then $\gr_{\mathcal{F}} A:= \bigoplus\fm^i/\fm^{i+1}$ is isomorphic to $A$.
\end{lemma}

Continuing with the above notation, define $\deg_a F_i A$ be
the the lowest Adams degree of a nonzero homogeneous element in $F_i A$.
A filtration $\mathcal{F} = \{F_i A : i \in \mathbb{Z} \}$ is called
{\it strict} if

\begin{enumerate}
\item[(F4)]
$\qquad \qquad \underline{\lim}_{n\to\infty} \frac{\deg_a F_{-n} A}{n}>0$.
\end{enumerate}
It is clear that the standard graded filtration is strict since
$\deg_a F_{-n} A\geq n$.

\begin{lemma}
\label{yylem1.3} Let $A$ be a connected graded algebra with
a graded filtration $\mathcal{F}= \{F_i A : i \in \mathbb{Z} \}$.
Let $B$ be the associated graded algebra
$\gr_{\mathcal{F}} A$.
\begin{enumerate}
\item[(1)]
If $B$ is a domain, so is $A$.
\item[(2)]
If $B$ is affine {\rm{(}}=finitely generated{\rm{)}} as an algebra, then
$\mathcal{F}$ is strict.
\end{enumerate}
\end{lemma}

\begin{proof} (1) This is standard.

(2) Assume that $B$ is generated by a finite set of homogeneous elements
$\{ \overline{f_i}\}_{i=1}^{w}$ where
$\{f_i\}$ are homogeneous elements in $A$ with $\deg_a f_i=a_i>0$ and
$f_i\in F_{b_i} A\setminus F_{b_i-1} A$. Let $\beta =\max\{ |b_i| \mid i=1,\cdots,w\}$
and $\alpha =\min\{a_i\mid i=1,\cdots, w\}$. We claim that
$\underline{\lim}_{n\to\infty} \frac{\deg_a F_{-n} A}{n}\geq \alpha/\beta >0$. As a consequence,
${\mathcal F}$ is strict.

Suppose $\deg_a F_{-n}A\neq \infty$ for some $n>0$.
Let $x\in F_{-n}A$ be a nonzero homogeneous element
such that $\deg_a x=\deg_a F_{-n}A$. We consider the following two
cases.

Case 1: $x\not\in F_{-n-1}A=F_{-(n+1)}A$. Then $\overline{x} \in B$, so we can write
$\overline{x}=p(\overline{f_1}, \cdots,
\overline{f_w})$ where $p$ is a noncommutative polynomial in $w$ variables.
Recalling that $\deg$ is the
bi-degree, this implies that $\deg \overline{x}=\deg (\overline{f_{i_1}}\cdots \overline{f_{i_m}})$
for some integers $1\leq i_s \leq w$ where $s=1,\cdots, m$. Thus we have
\begin{equation}
\label{E1.3.1}\tag{E1.3.1}
\deg_{\mathcal F} \overline{x}=-n=\sum_{s=1}^m \deg_{\mathcal F}
\overline{f_{i_s}},
\end{equation}
and
\begin{equation}
\label{E1.3.2}\tag{E1.3.2}
\deg_{a} \overline{x}=\sum_{s=1}^m \deg_{a}
\overline{f_{i_s}}.
\end{equation}

By \eqref{E1.3.1}, $n\leq \sum_{s=1}^m |\deg_{\mathcal F}
\overline{f_{i_s}}| \leq m \beta$; that is, $m\geq n/\beta$. By \eqref{E1.3.2}, we have
$$\deg_a x=\deg_a \overline{x}\geq m \alpha \geq \frac{n}{\beta} \alpha=(\alpha/\beta) n.$$

Case 2: $x\in F_{-(n+1)}A$. By (F2),
$\bigcap_{n>0} F_{-n} A=\{0\}$, so there is an $n_0>0$ such that
$x\in F_{-(n+n_0)}A\setminus F_{-(n+n_0+1)}A$. Note that $\deg_a x=
\deg_a F_{-(n+n_0)}A$. By Case 1, we have
$$\deg_a x\geq (\alpha/\beta) (n+n_0)\geq (\alpha/\beta) n.$$

Combining these two cases, we have
$\deg_a F_{-n}A=\deg_a x\geq (\alpha/\beta) n$ for all
$n>0$, proving the claim.
\end{proof}

\begin{proposition}
\label{yypro1.4} Let $A$ be a connected graded algebra with
a strict graded filtration $\mathcal{F}$. Let $B$ be the associated
graded algebra $\gr_{\mathcal{F}} A$.
\begin{enumerate}
\item[(1)]
$B$ is {\rm{(}}left{\rm{)}} noetherian if and only if
$\mathrm{Rees}_{\mathcal{F}} A$ is.
\item[(2)]
If $B$ is {\rm{(}}left{\rm{)}} noetherian, then so is $A$. In this case $A$ is
locally finite and affine.
\item[(3)]
If $B$ is noetherian Auslander Gorenstein and GK-Cohen-Macaulay, then
so is $A$.
\item[(4)]
If $B$ as in {\rm{(3)}} has finite global dimension, then so does $A$.
In this case
\begin{equation}
\label{E1.4.1}\tag{E1.4.1} \gldim A=\GKdim A = \GKdim B = \gldim B.
\end{equation}
\end{enumerate}
\end{proposition}

\begin{proof} (1) By (F4), there is a positive integer $d$ such that
$\deg_a F_{-n} A> (n+1)/d$ for all $n>0$. For the rest of the proof we
change the Adams grading of $A$ by setting a new Adams grading
$$\widetilde{\deg_a} x:=d\deg_a x$$  for a homogeneous element $x$.
Under the new Adams grading, we have $\deg_a F_{-n}A>n$ for all $n>0$.
Consider the Rees ring $R:=\bigoplus_{i=-\infty}^{\infty} r^{i} F_{i} A$
with the total grading. In particular, one has
$\deg r=1$. Then $R$ is a connected graded algebra and $r$ is a central regular
element in $R$ of degree 1. It is clear that $R/rR\cong B$. The assertion
follows from \cite[Proposition 3.5(a)]{Le}.

(2) Suppose that $B$ is left noetherian. As noted in the proof of (1), $R$ is connected
graded with respect to the total grading. By part (1), $R$ is left noetherian. Since $A$
is isomorphic to $R/(r-1)R$, $A$ is left noetherian. Noetherian
connected graded algebras are locally finite and affine, and hence the second
assertion follows.

(3) Suppose that $B$ is as stated. As in the proofs of parts (1) and (2),
$B\cong R/rR$ and $A\cong R/(r-1)R$, where $r$
is the central regular element of degree 1 in $R$ introduced in (1). Thus $R$ is
Auslander Gorenstein and GK-Cohen-Macaulay, by
\cite[Theorem 5.10]{Le}. By the Rees Lemma \cite[Proposition 3.4(b)
and Remarks 3.4(3)]{Le}, $A$ is Auslander Gorenstein and
GK-Cohen-Macaulay.

(4) Suppose that $\mathrm{gldim}B < \infty$, with $B$ noetherian, 
Auslander-Gorenstein and GK-Cohen-Macaulay. By parts (2) and (3),
$A$ is noetherian,
Auslander Gorenstein, and GK-Cohen-Macaulay. Moreover $B$ is a
domain by \cite[Theorem 4.8]{Le}, so $A$ is a domain by
Lemma \ref{yylem1.3}(1). It remains to prove (\ref{E1.4.1}).
Since $B=R/rR$ with $r$ a central non-zero-divisor of degree 1,
\cite[Lemma 7.6]{LP} implies that $R$ has global
dimension equal to $\gldim B+1$. Let $C$ be the localization $R[r^{-1}]$.
Then
\begin{equation}
\label{E1.4.2}\tag{E1.4.2}\gldim C\leq \gldim R=\gldim B+1.
\end{equation}
Note that $C=A[r^{\pm 1}]$, so that $\gldim A + 1 = \gldim C$.
Therefore, by \eqref{E1.4.2}, $\gldim A \leq \gldim B$. Since
both $A$ and $B$ are
affine, one can use their Hilbert series to compute their
GK-dimensions, \cite[Lemma 6.1(b)]{KL}; hence
$\GKdim A\geq \GKdim B$. By the GK-Cohen-Macaulay condition for
connected graded algebras, $\GKdim A=\gldim A$ and $\GKdim B=\gldim B$.
Therefore (\ref{E1.4.1}) follows.
\end{proof}

\section{Hopf algebras whose algebra is connected graded}
\label{yysec2}

\subsection{Shifting the augmentation}
\label{yysec2.1}

\begin{lemma}
\label{yylem2.1} Let $H$ be a Hopf algebra.
\begin{enumerate}
\item[(1)]
Let $I$ be the kernel of a character $\chi: H \rightarrow k$. Then
there is an algebra isomorphism $\sigma\in {\rm{Aut}}(H)$ such
that $\sigma(I)=\ker \epsilon$.
\item[(2)]
If $H$ is connected graded as an algebra, then there is a
grading of $H$ such that $H=\bigoplus_{i=0}^\infty H(i)$ is
connected graded and $\ker \epsilon =\bigoplus_{i\geq 1}H(i)$.
\end{enumerate}
\end{lemma}

\begin{proof}
(1) Take $\sigma$ to be the left winding automorphism $h\mapsto
\sum \chi(h_1) h_2$.

(2) Let $H=k\oplus \bigoplus_{i=1}^\infty B(i)$ be the given
grading of $H$ and let $I=\bigoplus_{i=1}^\infty B(i)$.
By part (1), there is an algebra automorphism $\sigma$ of $A$ such that
$\sigma(I)=\ker \epsilon$. Let $H(i)=\sigma(B(i))$. Then
$H=\bigoplus_{i=0}^\infty H(i)$ is a connected graded algebra
and $\ker \epsilon =\bigoplus_{i\geq 1}H(i)$.
\end{proof}

\subsection{Structure of connected graded Hopf algebras}
\label{yysec2.2}
We now prove Theorem \ref{yythm0.2}.

Let $A$ be an algebra and $M$ an $A$-bimodule. For any $i\geq 0$,
the $i^{th}$ Hochschild homology (respectively, Hochschild cohomology)
of $A$ with coefficients in $M$ is denoted by $H_i(A,M)$
(respectively, $H^i(A,M)$).  We say that \emph{Poincar{\'e}
duality} holds for Hochschild (co)homology over $A$ if there is a
non-negative integer $n$ such that
$$H^{n-i}(A,M)\cong H_{i}(A,M)$$
for all $i$, $0\leq i \leq n,$ and for all $A$-bimodules $M$.

\begin{theorem}
\label{yythm2.2}
Let $H$ be a Hopf $k$-algebra which is connected graded as an algebra.
In parts {\rm{(2,3,4,5)}}, we further assume that $\GKdim H =n <\infty$.
\begin{enumerate}
\item[(1)]
$H$ is a domain.
\item[(2)]
$H$ is affine as an algebra,
and is noetherian, Cohen-Macaulay, Auslander regular and Artin-Schelter
regular, with $\gldim H = n$.
\item[(3)]
$H$ is Calabi-Yau; that is, the Nakayama automorphism of $H$ is the identity.
\item[(4)]
$S^2=Id_H$.
\item[(5)]
Poincar{\'e} duality holds for Hochschild (co)homology over $H$.
\end{enumerate}
\end{theorem}

\begin{proof}
(1) By Lemma \ref{yylem2.1}(2), we may assume that $\fm :=
\bigoplus_{i \geq 1}H(i)
=\ker \epsilon$. Let $\mathcal{F}$ be the standard graded filtration as in
\eqref{E1.1.1}, and write $\gr_{\mathcal{F}} H$ as $\gr H$. Then
$\gr H\cong \bigoplus_{i=0}^{\infty} \fm^i/\fm^{i+1}$.
By \cite[Proposition 3.4(a)]{GZ}, as Hopf algebras,
\begin{equation}
\label{E2.2.1}\tag{E2.2.1}
\gr H \quad \cong \quad U({\mathfrak g})\end{equation}
for some positively graded Lie algebra $\mathfrak g$, where, in
\eqref{E2.2.1},
$U(\mathfrak{g})$ has its standard cocommutative coalgebra structure. Since
$U({\mathfrak g})$ is a domain, so is $H$ by Lemma \ref{yylem1.3}(1).

(2) By \cite[Proposition 3.4(b)]{GZ}, $\dim {\mathfrak g}\leq
\GKdim H<\infty$. So ${\mathfrak g}$ is finite dimensional.
Then $\gr H$ is affine.
Since ${\mathfrak g}$ is finite dimensional,
$\gr H$ is a noetherian Auslander regular and
Cohen-Macaulay domain. The same list of properties for $H$,
and with them the fact that $\gldim H = n$, follows from
Proposition \ref{yypro1.4}(3) and (4). Since $H$ is a noetherian
connected graded algebra, it is easily seen to be affine.
Finally, $H$ is Artin-Schelter regular by \cite[Lemma 6.1]{BZ}.

(3,4) By (2) and \cite[Corollary 0.4]{BZ}, $H$ is skew
Calabi-Yau, so it remains for (3) to show that the Nakayama
automorphism is trivial. Since the bimodule
$\int^{\ell}_H := \Ext^n_H(H/\fm, H)$ is graded and has
dimension one by Artin-Schelter regularity (2), it is the trivial module
(on both right and left). Hence, by \cite[Theorem 0.3]{BZ}, the
Nakayama automorphism of $H$ is $S^2$. Note that, since $H$ is
a connected graded domain, it has no non-trivial inner automorphisms.
Hence, by \cite[Theorem 0.6]{BZ}, $S^4=Id_H$. It remains to
show that $S^2 =Id_H$.

By \eqref{E2.2.1}, $S_{\gr H}^2=Id_{\gr H}$. Moreover, since
$\gr H$ is the enveloping algebra of a nilpotent Lie algebra, the
Nakayama automorphism $\mu_{\gr H}$ is the identity. Let
$\mu_R$ be the Nakayama automorphism of
$R:=\mathrm{Rees}_{\mathcal{F}} H$.  Keeping the notation
introduced for $R$ in $\S$\ref{yysec1.3}, since $r$ is central,
\begin{equation}
\label{E2.2.2}\tag{E2.2.2}
\mu_{\gr H}=\mu_R\otimes_R \gr H =\mu_R\otimes_R R/\langle r \rangle,
\end{equation}
by \cite[Lemma 1.5]{RRZ}. Similarly, since \cite[Lemma 1.5]{RRZ}
holds in an ungraded setting (with the same proof),
\begin{equation}
\label{E2.2.3}\tag{E2.2.3}
\mu_H=\mu_R\otimes_R H=\mu_R\otimes_R R/\langle r-1 \rangle.
\end{equation}
Since $H$, $\gr H$ and $R$ are all graded, their Nakayama automorphisms
preserve their gradings. Since $r$ is central in $R$, $\mu_R(r)=r$.
Suppose $\mu_H$ is not the identity. Nevertheless $\mu_H^2=Id_H$ by the
previous paragraph. Pick a nonzero homogeneous element $x \in H$ of smallest
degree such that $\mu_H (x)=-x$. Fix $i$ such that
$x\in \fm^i\setminus \fm^{i+1}$. Since $\mu_R$ preserves the grading, $y:=\mu_R(x)\in H$
has the same degree as $x$. So $\mu_R(xr^{-i})=yr^{-i}$
where $y\in \fm^i\setminus \fm^{i+1}$. Write $w'=\gr_{\mathcal{F}} w$ for every $w\in H$.
Then \eqref{E2.2.2} implies that
\begin{align*} x'=& \mu_{\gr H}(x') \\
=& \mu_R(xr^{-i}) \mod (r) \\
=& yr^{-i} \mod(r) \\
=& y'.
\end{align*}
Hence $x-y\in \fm^{i+1}$, so $y=x+z$ for some $z\in \fm^{i+1}$.
Similarly, \eqref{E2.2.3} implies that $y=-x$ as $\mu_H(x)=-y$.
Therefore, $-x=x+z$, so that $x=-\frac{1}{2} z
\in \fm^{i+1}$. This is a contradiction, and so $\mu_H=S^2 = Id_H$ as required.

(5) This follows from \cite[Corollary 0.4]{BZ} and the fact (3) that
the Nakayama automorphism $\mu_H$ is the identity.
\end{proof}

Theorem \ref{yythm2.2}(2) answers \cite[Question C]{Br1},
\cite[Question E]{Br2} and \cite[Question 3.5]{Go2} affirmatively when $H$
is connected graded of finite Gel'fand-Kirillov dimension as an algebra.
Theorem \ref{yythm0.2} follows from the above theorem.
The following are convenient reformulations of parts of Theorem \ref{yythm2.2}.

\begin{corollary}
\label{yycor2.3} Let $A$ be an algebra of finite GK-dimension that
is not Calabi-Yau.
\begin{enumerate}
\item[(1)]
If $A$ is connected graded, then $A$ does not possess a
Hopf algebra structure.
\item[(2)]
If $A$ has a Hopf algebra structure, then $A$ cannot be connected
graded as an algebra.
\end{enumerate}
\end{corollary}

\subsection{Hilbert series}
\label{yysec2.3}
The following is Theorem \ref{yythm0.3}.

\begin{theorem}
\label{yythm2.4}
Let $H$ be a Hopf algebra that is connected graded and locally
finite as an algebra.
\begin{enumerate}
\item[(1)]
The Hilbert series of $H$ is
$$\frac{1}{\prod_{i=1}^\infty (1-t^i)^{n_i}}$$
for some non-negative integers $n_i$. Further, the sequence $\{n_i\}_{i\geq 1}$
is uniquely determined by the Hilbert series of $H$.
\item[(2)]
Suppose $H$ is generated in degree 1. Then $H$ is
isomorphic as an algebra to $U({\mathfrak g})$ for a graded Lie algebra
${\mathfrak g}$ generated in degree 1. Hence, for all $i \geq 1$,

$$\mathrm{dim}_k \mathfrak{g}_i \quad = \quad n_i,$$
where $\mathfrak{g}_i$ is the $i^{\mathit{th}}$ degree component
of ${\mathfrak g}$.
\item[(3)] Suppose again that $H$ is generated in degree 1, $H \neq k$. Then
$$ \{j: n_j \neq 0 \} \quad = \quad [1,\ell) \cap \mathbb{N},
\textit{ for some } \ell \in [2, \infty]. $$
\item[(4)]
If the Hilbert series of $H$ is $\frac{1}{(1-t)^n}$ for some $n \geq 1$,
then $H$ is commutative, and hence $H \cong k[x_1, \ldots , x_n].$
\item[(5)]
$\mathrm{GKdim}H < \infty$ if and only if $\sum_{i=1}^{\infty} n_i<\infty$.
In this case, $$\GKdim H=
\sum_{i=1}^{\infty} n_i.$$
\end{enumerate}
\end{theorem}

\begin{proof} (1),(2): Let $\mathcal{F}$ be the standard graded filtration.
Since $H$ is locally finite, $\gr_{\mathcal{F}} H$ and $H$
have the same Hilbert series with respect to their Adams grading.
Thus, in proving (1), we may assume that $H=\gr_{\mathcal{F}} H$. By
\cite[Proposition 3.4(a)]{GZ}, $\gr_{\mathcal{F}} H$ is isomorphic to
$U({\mathfrak g})$ for a graded Lie algebra ${\mathfrak g}$,
with ${\mathfrak g}$ generated by ${\mathfrak g}_1:=\fm/\fm^2$.
Since both $\fm$ and $\fm^2$ are Adams graded, so is
${\mathfrak g}_1$. Therefore ${\mathfrak g}$ is Adams graded.
Since $H$ is locally finite, so is ${\mathfrak g}$. Now, by the PBW theorem,
$$H_{H}(t)=H_{U({\mathfrak g})}(t)=
\frac{1}{\prod_{i=1}^\infty (1-t^i)^{n_i}}$$
where $n_i$ is the dimension of the degree $i$ component $\mathfrak{g}_i$
of ${\mathfrak g}$. It is clear that the sequence $\{n_i\}_{i \geq 1}$ is
uniquely determined by
the Hilbert series of $H$.

Now (2) follows from the above together with Lemma \ref{yylem1.2}.

(3) Suppose that $H$ is generated in degree 1. By (2), so is $\mathfrak{g}$.
Now (3) is a consequence of (2) and the fact that, for all $i > 1$,
$\mathfrak{g}_i = [\mathfrak{g}_1,\mathfrak{g}_{i-1}]$.

(4) Suppose that the Hilbert series of $H$ is $(1-t)^n$. Then (1)
implies that $n=n_1$ and $n_i=0$ for all $i>1$. Hence, by the
proof of (1), ${\mathfrak g}$ is concentrated in degree 1. Thus ${\mathfrak g}$
is abelian, and $\gr_{\mathcal{F}} H$ is generated in degree 1. As a consequence,
$H$ also is generated in degree 1. By Lemma \ref{yylem1.2},
$H\cong \gr_{\mathcal{F}} H=U({\mathfrak g})$, which is commutative.

(5) Suppose that $\sum_{i=1}^{\infty} n_i<\infty$. By (1), $\GKdim H<\infty$.
By Theorem \ref{yythm2.2}, $H$ is an affine domain, so that $\GKdim H$
can be computed by its Hilbert series, by \cite[Lemma 6.1(b)]{KL}. Hence,
$\GKdim H=\sum_{i=1}^{\infty} n_i$ by \cite[Theorem 12.6.2]{KL}.

Suppose conversely that $\mathrm{GKdim}H < \infty$. Then
$\mathrm{GKdim}\gr_{\mathcal{F}} H < \infty$ by
\cite[Lemma 6.5]{KL}. Since $\gr_{\mathcal{F}} H
\cong U(\mathfrak{g})$ as noted in the proof of (1),
$\mathrm{dim}_k \mathfrak{g} = \sum_i n_i < \infty$ by \cite[Example 6.9]{KL}.
\end{proof}

\begin{remarks}
\label{yyrem2.5}
(1) It is trivial but nevertheless perhaps relevant to observe
that the hypothesis of degree 1 generation in parts (2) and (3)
of the theorem is not always valid: for example, the coordinate
ring of the 3-dimensional Heisenberg group $U$ is a graded Hopf
algebra $\mathcal{O}(U) \cong k[X,Y,Z]$, where the generators
have degrees 1,1 and 2.

(2) By part (2) of the above theorem, when $H$ is generated in
degree 1, then $H$ is isomorphic as an algebra to the enveloping
algebra of a nilpotent Lie algebra $\mathfrak{g}$.
This is not in general true, however, when the
hypothesis of degree 1 generation is
dropped - see Theorem \ref{yythm0.5}(2) for an example.
\end{remarks}

\subsection{Another proof of Theorem \ref{yythm0.3}(4)}
\label{yysec2.4}
We offer in this subsection an alternative proof of Theorem
\ref{yythm0.3}(4), which will motivate the generalisation 
in $\S$3. We use here the notation introduced in
$\S$\ref{yysec1.3}.

\begin{lemma}
\label{yylem2.6} Let $H$ be an affine Hopf algebra.
Suppose $H$ is connected graded as an algebra.
\begin{enumerate}
\item[(1)]
The abelianization $H_{ab}$ is
isomorphic to a commutative polynomial ring.
\item[(2)]
$ d(H) \leq \mathrm{GKdim}H$ and $d(H) \leq \mathrm{Kdim}H.$
\item[(3)]
Suppose that $ d(H) =\mathrm{GKdim}H$ or that $d(H) =
\mathrm{Kdim}H.$ Then $H=H_{ab}$.
\end{enumerate}
\end{lemma}

\begin{proof} (1) This follows from Theorem \ref{yythm0.1}(3) since
$H_{ab}$ is commutative and connected graded.

(2) As noted in Lemma \ref{yylem1.1}, since $H$ is connected graded,
\begin{equation}
\label{E2.6.1}\tag{E2.6.1}
d(H_{ab}) \quad = \quad d(H).
\end{equation}
Since $H_{ab}$ is a  polynomial ring by part (1),
\begin{equation}
\label{E2.6.2}\tag{E2.6.2}
d(H_{ab}) = \mathrm{GKdim}H_{ab} = \mathrm{Kdim}H_{ab}.
\end{equation}
The assertion follows from \eqref{E2.6.1} and \eqref{E2.6.2}.

(3) Suppose that $d(H) = \GKdim H$. Then $\GKdim H_{ab} = \GKdim H$
by \eqref{E2.6.1} and \eqref{E2.6.2}. Since $H$ is a domain by
Theorem \ref{yythm2.2}(1), $H = H_{ab}$ by \cite[Proposition 3.15]{KL}.
The argument for Krull dimension is similar.
\end{proof}

\begin{proof}[Proof of Theorem {\rm{\ref{yythm0.3}(4)}}]
Replace $H$ by $A$.
Suppose that $H_A(t)=\frac{1}{(1-t)^{n}}$. Then
\begin{equation}
\label{E2.6.4}\tag{E2.6.4}
 d(A) \geq \mathrm{dim}_k A_1 = n = \mathrm{GKdim}A.
\end{equation}
Thus $A = A_{ab}$ by Lemmas \ref{yylem2.6}(2),(3), and so $A$
is a polynomial algebra by Lemma \ref{yylem2.6}(1).
\end{proof}

\subsection{Consequences}
\label{yysec2.5}
We give two ``no Hopf structure" results. The first is a
straightforward consequence of \cite[Proposition 3.4(a)]{GZ},
which we record here since it fits the present context.
The second assembles some immediate consequences of
Theorem \ref{yythm0.3}.

\begin{proposition}
\label{xxpro2.7}
Let $A$ be a connected graded algebra and let $x,y$ be
nonzero elements in $A$ such that $xy=q yx$ for some
$q \in k\setminus \{1\}$. Then there is no Hopf
algebra structure on $A$.
\end{proposition}

\begin{proof} Suppose $A$ is a Hopf algebra. Passing to
the associated graded ring associated to the standard
filtration, we may assume that $A=U({\mathfrak g})$ for
a graded Lie algebra $\mathfrak{g}$, by \cite[Proposition 3.4(a)]{GZ}.
Taking a Hopf subalgebra, we might assume that
${\mathfrak g}$ is locally finite, and then, factoring by a
suitable Lie ideal, we may assume that ${\mathfrak g}$ is
finite dimensional. Now there is a
connected ${\mathbb N}$-filtration ${\mathcal F}$
associated to the Lie algebra such that $\gr_{\mathcal F} A$
is isomorphic to $S(\mathfrak{g})$, a commutative
polynomial ring. But this is impossible if there are
nonzero elements $x,y \in A$ and $q\neq 1$, such that $xy=qyx$.
\end{proof}

\begin{corollary}
\label{yycor2.8}
There is no Hopf algebra structure on
the following Koszul Artin-Schelter regular algebras:
\begin{enumerate}
\item[(1)]
Sklyanin algebras of any dimension;
\item[(2)]
skew polynomial rings $k_{p_{ij}}[x_1,\cdots,x_n]$ {\rm{(}}except for the
case $p_{ij}=1$ for all $i,j${\rm{)}};
\item[(3)]
quantum matrix algebras
${\mathcal O}_{q}(M_{n\times n})$ {\rm{(}}except for $q=1${\rm{)}}.
\item[(4)]
noncommutative Koszul Artin-Schelter regular algebras of
dimension $\leq 4$.
\end{enumerate}
\end{corollary}

\begin{proof} We claim that the algebras listed
above have Hilbert series $(1-t)^{-d}$ for some $d$.
For parts (1,2,3), this is clear. For part (4),
this follows from \cite[Proposition 1.4]{LP}.
Now the assertion follows from Theorem
\ref{yythm0.3}(4).
\end{proof}

\subsection{The theorem of Latysev-Passman for connected
graded algebras}
\label{yysec2.6}
Latysev \cite{Lat} proved that the enveloping algebra
$U(\mathfrak{g})$ of a finite dimensional Lie algebra
$\mathfrak{g}$  over a field of characteristic 0 satisfies
a polynomial identity (or we say $U(\mathfrak{g})$ is PI)
if and only if $\mathfrak{g}$ is abelian. The hypothesis
of finite dimensionality was removed by Passman \cite{Pas}.
The aim of this subsection is to show that the same result
applies to Hopf algebras which are connected graded as algebras.

\begin{theorem}
\label{yythm2.9} Let $H$ be a Hopf algebra. Suppose that $H$ is
connected graded or that $\bigcap_{i=1}^{\infty} \fm^i=\{0\}$
where $\fm=\ker \epsilon$. Then $H$ is PI if and only if $H$
is commutative.
\end{theorem}

\begin{proof} Note that, in the light of Lemma \ref{yylem2.1}, if
$H$ is connected graded as an algebra then
$\bigcap_{i=1}^{\infty} \fm^i=\{0\}$. Hence we assume
that $\bigcap_{i=1}^{\infty} \fm^i=\{0\}$.

Let ${\mathcal F}$ be the standard filtration, as defined
in \eqref{E1.1.1} in $\S$\ref{yysec1.3}. If $H$ is PI, then so is
$\gr_{\mathcal F} H$. By \cite[Proposition 3.4(a)]{GZ},
$\gr_{\mathcal F} H \cong U({\mathfrak g})$. Hence, by \cite{Pas},
$\gr_{\mathcal F} H$ is commutative. By \cite[Lemma 3.5]{GZ},
$H$ is commutative, as required.
\end{proof}

\begin{remark}
\label{yyrem2.10}
The above theorem fails when $k$ has positive characteristic.
Consider, for example, the enveloping algebra $H$ of the
3-dimensional Heisenberg Lie algebra over a field of
positive characteristic. Then $H$ is connected graded as
an algebra, and is a finite module over its centre and is
therefore PI. However $H$ is not commutative.
\end{remark}

\section{Almost commutative Hopf algebras}
\label{yysec3}

At this point we take a small diversion to record a
variation of the strand initiated in Theorem \ref{yythm2.9},
using similar ideas to those employed above,  and making
further use of \cite[$\S3$]{GZ}. The outcome can be viewed
as an un-graded version of Lemma \ref{yylem2.6}. Unlike in
most of the rest of the paper, there is no connectedness
hypothesis on $H$ in the following results; however we
continue to assume that $k$ is algebraically closed of
characteristic 0.

\begin{proposition}
\label{yypro3.1}
Let $H$ be a Hopf $k$-algebra with $\GKdim H = n < \infty$, and
let $\mathfrak{m} = \mathrm{ker} \epsilon$.
\begin{enumerate}
\item[(1)]
$\dim \fm/\fm^2 \leq \GKdim H$.
\item[(2)]
Suppose that $H$ is affine or noetherian.
Then the following statements are equivalent:
\begin{enumerate}
\item[(a)]
$\dim \fm/\fm^2 = \GKdim H$.
\item[(b)]
$\mathfrak{m}$ contains a unique minimal prime $P$ of $H$, 
with $H/P \cong \mathcal{O}(G)$ for a connected
algebraic group $G$ of dimension $n$.
\item[(c)]
$H$ has a commutative factor algebra $A$ with $\GKdim A = n$.
\end{enumerate}
\end{enumerate}
\end{proposition}

\begin{proof} (1)
Let $\mathcal{F}$ be the filtration $\{\fm^i : i \geq 0 \}$ of $H$,
so that $\gr_{\mathcal{F}} H = \oplus_{i \geq 0} \fm^i/\fm^{i+1}$.
Then $\gr_{\mathcal{F}} H \cong U(\mathfrak{g})$ for some Lie
algebra $\mathfrak{g}$ by \cite[Proposition 3.4(a)]{GZ}. Moreover,
$$\begin{aligned}
\GKdim H \; &\geq \; \dim_k \mathfrak{g} \quad \quad \quad
&\textrm{\cite[Proposition 3.4(b)]{GZ}} \\
&\geq  \dim_k \fm / \fm^2, \quad \quad
&\textrm{since } \fm/ \fm^2 \hookrightarrow \mathfrak{g},
\end{aligned}
$$
where the inclusion in the last line is given by \cite[Lemma 3.3(d)]{GZ}.

(2) $(b) \Rightarrow (c)$: Trivial.

$(a) \Rightarrow (b)$: Suppose that $\dim_k \fm / \fm^2 = n.$  Set
$J_{\fm} := \bigcap_{i \geq 0} \fm^i$. Then
\begin{equation}
\label{E3.1.1}\tag{E3.1.1}
\gr_{\mathcal{F}}H \cong U(\mathfrak{g}) \cong \gr_{\mathcal{F}}(H/J_{\fm}),
\end{equation}
and
\begin{equation}
\label{E3.1.2}\tag{E3.1.2}
\dim_k \mathfrak{g} = \GKdim \gr_{\mathcal{F}} H \leq \GKdim H = n,
\end{equation}
by \cite[Lemma 6.5 and Example 6.9]{KL}. By \cite[Lemma 3.3(d)]{GZ},
$\fm / \fm^2 \subseteq \mathfrak{g}$, so that, from (a) and
\eqref{E3.1.2},
$$ \GKdim \gr_{\mathcal{F}} H = n. $$
Since $\GKdim H/J_{\fm} \geq \GKdim \gr_{\mathcal{F}} H $ by
\eqref{E3.1.1} and \cite[Lemma 6.5]{KL}, we deduce that
\begin{equation}
\label{E3.1.3}\tag{E3.1.3} \GKdim H = \GKdim H/J_{\fm}.
\end{equation}
Moreover, $H/J_{\fm}$ is a domain, since its associated graded
algebra $\gr_{\mathcal{F}} H/J_{\fm}$ is a domain and
$\mathcal{F}$ is separating on $H/J_{\fm}$. Hence, given
\eqref{E3.1.3}, $J_{\fm}$ is a minimal prime ideal of $H$.

Now $H/J_{\fm}$ is a Hopf algebra by \cite[Lemma 4.7]{LWZ},
so that $(H/J_{\fm})_{ab}$ is also a Hopf algebra by
\cite[Lemma 3.7]{GZ}. If $H$ is noetherian, then so is
$(H/J_{\fm})_{ab}$, and hence this commutative Hopf algebra
is affine by Molnar's theorem \cite{Mol}. Thus, since $k$
has characteristic 0, and whether $H$ is affine or noetherian,
$(H/J_{\fm})_{ab}$ is the coordinate ring of an algebraic
group $G$, \cite{Wat}; and as such, $(H/J_{\fm})_{ab}$  has
finite global dimension,
\cite[\S11.4 and \S11.6]{Wat}. More precisely, by, for
example, \cite[Theorem 5.2]{Hu},
\begin{equation}
\label{E3.1.4}\tag{E3.1.4}
\dim G = \GKdim \, (H/J_{\fm})_{ab} = \gldim \, (H/J_{\fm})_{ab}
= \dim_k \mathfrak{n}/\mathfrak{n^2},
\end{equation}
where $\mathfrak{n}$ denotes the augmentation ideal of
$(H/J_{\fm})_{ab}$. But, by the definition of $(H/J_{\fm})_{ab}$,
we have
$$ \mathfrak{n}/\mathfrak{n}^2 = \fm/\fm^2.$$
Therefore, invoking hypothesis (a), all the dimensions in
\eqref{E3.1.4} are equal to $n$. Now $J_{\fm}$ is a prime
ideal of $H$ and $\GKdim H/J_{\fm} = n$ by (a) and
\eqref{E3.1.3}, so $H/J_{\fm}$ is an Ore domain by
\cite[Theorem 4.12]{KL}. Hence, the proper factors of
$H/J_{\fm}$ have Gel'fand-Kirillov dimension strictly
less than $n$ by \cite[Proposition 3.15]{KL}, and we thus
deduce that
$$ H/J_{\fm} \; = \; (H/J_{\fm})_{ab} \; = \; \mathcal{O}(G).$$
That is, $G$ is connected, and (b) is proved, with $P = J_{\fm}$.

$(c)\Rightarrow (a)$: Suppose that $H$ has a commutative factor
algebra $C$ with $\GKdim C = n.$ By (1), it is enough to prove
that $\dim_k \fm/ \fm^2 \geq n.$ Thus we only need to prove
the result for the factor $H_{ab}$ of $H$, which is a Hopf
algebra by \cite[Lemma 3.7]{GZ}. For commutative Hopf algebras
the noetherian and affine hypotheses coincide, by the Hilbert
basis theorem and Molnar's theorem \cite{Mol}, so we can
assume that $H = \mathcal{O}(G)$ is a commutative affine Hopf
algebra with $\GKdim H = n < \infty$. Now $J_{\fm} = 0$ without
loss of generality since we can pass to a suitable prime Hopf
factor of $H$ (namely, the coordinate ring of the connected
component of $1_G$), and in this factor domain Krull's
intersection theorem applies. So by standard commutative
algebra (or see \cite[Proposition 3.6]{GZ}),
$\dim_k \fm/\fm^2 = n$, as required.
\end{proof}

As an immediate consequence of the above proposition we have

\begin{corollary}
\label{yycor3.2}
Let $H$ be an affine prime Hopf algebra of finite GK-dimension.
Suppose that $\dim \fm/\fm^2\geq \GKdim H$.
Then $H$ is commutative.
\end{corollary}

\begin{remark}
\label{yyrem3.3}
Proposition \ref{yypro3.1} prompts an obvious problem, namely:
to completely describe the Hopf algebras featuring in part (2)
of the proposition. We leave this for the future.
\end{remark}

\section{Hopf algebras that are connected as a coalgebra}
\label{yysec4}

In this section we study a connected Hopf algebra $H$ of finite
Gel'fand-Kirillov dimension $n$, writing $\GKdim H = n$.
The associated graded ring with respect to the
coradical filtration of $H$, whose definition was
recalled in $\S$\ref{yysec1.2}, is denoted by $\gr_c H$.
By \cite[Theorem 6.9 and Corollary 6.10]{Zh},
\begin{enumerate}
\item[(1)]
$H$ and $\gr_c H$ are affine noetherian domains;
\item[(2)]
$\gr_c H$ is isomorphic to a commutative
polynomial ring in $n$ indeterminates;
\item[(3)]
$H$ is Auslander regular of global dimension $n$;
\item[(4)]
$H$ is GK-Cohen-Macaulay.
\end{enumerate}
From (1),(3) and \cite[Lemma 6.1]{BZ}, $H$ is Artin-Schelter regular.
Therefore, by \cite[Theorem 0.3]{BZ},
$H$ is skew Calabi-Yau (CY), with Nakayama automorphism given by
$S^2\circ \Xi^{l}_{\chi}$, where $\Xi^{l}_{\chi}$ denotes the
left winding automorphism of $H$ arising from the right character
$\chi$ of the left integral $\int_H^l$ of $H$.

\subsection{Artin-Schelter regularity of the coalgebra structure}
\label{yysec4.1}
The main result of this
subsection concerns the Artin-Schelter regularity
(or skew CY property) of the coalgebra structure of
$H$. The Artin-Schelter regularity of an artinian coalgebra is
introduced in \cite[Definition 2.1 and Remark 2.2]{HT}. We
now prove Theorem \ref{yythm0.9}. Recall the definition of a
{\it graded Hopf algebra} from $\S$\ref{yysec0.4}. If $A$ is 
a locally finite graded Hopf algebra, then the graded
$k$-linear dual of $A$, denoted by $A^{\circ}$, is also a
graded Hopf algebra.

\begin{theorem}
\label{yythm4.1}
Let $H$ be a connected Hopf algebra with GK-dimension $n < \infty$.
Then, as a coalgebra, $H$ is artinian and Artin-Schelter regular
of global dimension $n$.
\end{theorem}

\begin{proof} First we assume that $H=\gr_c H$,
which is a graded Hopf algebra and is connected graded as an algebra.
Let $H^\circ$ be the graded dual of $H$. By \cite[Theorem 6.9]{Zh},
$H$ is commutative, so $H^\circ$ is a cocommutative Hopf algebra
which is connected graded as an algebra. Consequently, $H^\circ$ is
isomorphic to $U(\mathfrak g)$ for a Lie algebra $\mathfrak{g}$ of
dimension $n$, which is positively graded, since $H^\circ$ is
generated in degree 1 by \cite[Lemma 5.5]{AS2}. (Also see the
discussion in subsection $\S$0.2). Thus, for example, by the discussion
in the opening paragraph of $\S 3$, $H^\circ$ is a noetherian algebra
of global dimension
$n$ satisfying the Artin-Schelter property. By duality, or a graded
version of \cite[Propositions 1.2 and 2.3]{HT}, it follows that $H$ is an
artinian coalgebra with global dimension $n$, satisfying the following
Artin-Schelter condition, see \cite[Definition 2.1 and Remark 2.2]{HT}
and Definition \ref{yydef4.3} below:
\begin{equation}
\label{E4.1.1}\tag{E4.1.1}
{\text{Ext}}^i_{H-comod}(H,k)=\begin{cases} 0 & i\neq n,\\
k& i=n.\end{cases}
\end{equation}
Here $H-comod$ (respectively, $comod-H$) denotes the category of
left (respectively, right) $H$-comodules.

Now, for general $H$ as in the theorem, let $K=\gr_c H$. By the
first paragraph, the coalgebra
$K$ is artinian of global dimension $n$. Hence
\footnote{This is dual to the following fact in the algebra
setting: if $A$ is complete and
$\gr_{\fm} A$ is noetherian and has finite global dimension, then
$A$ is noetherian and has finite global dimension
\cite[Theorem I.5.7 and $\S$I.7.2, Corollary 2]{LvO1}.}, so is $H$.
Since $K$ has global dimension $n$, the primitive cohomology dimension of
$H$ is no more than $n$ by \cite[Lemma 8.4(1)]{Wa2}. Since $H$ is
connected and artinian, the primitive cohomology dimension of $H$ equals
the global dimension of $H$ by \cite[Corollary 3]{NTZ}.

It remains to show that $H$ satisfies the Artin-Schelter condition
\eqref{E4.1.1}. Let $H^*$ be the complete ring $\lim_{i\to\infty}
(H_i)^*$, where $\{H_i : i \geq 0\}$ is the coradical filtration of $H$.
Then $H^*$ is a local noetherian algebra \cite[Proposition 1.1]{HT}.
Further, $H^*$ is a filtered algebra with $\fm$-adic filtration
where $\fm$ is the maximal ideal of $H^*$. It is easy to check
that $\gr_{\mathfrak m} H^*$ is the graded dual of $K(=\gr_c H)$, which is denoted
by $K^\circ$. Since $K^\circ$ is noetherian, the $\fm$-filtration
of $H^*$ is a Zariskian filtration  \cite[$\S$II.2.2, Proposition 1]{LvO1}.
By the first paragraph, $K^\circ\cong U(\mathfrak g)$, which is Auslander
regular. Hence $H^*$ is Auslander regular \cite[$\S$III.2.2, Theorem 5]{LvO1}. Since
$\Ext^n_{K^\circ}(k,K^\circ)=k$, which is one-dimensional,
$\Ext^n_{H^*}(k, H^*)$ is one-dimensional by \cite[Theorem 4.7]{LvO2}.
By a local version of
\cite[Theorem 6.3]{Le} (the proof of \cite[Theorem 6.3]{Le} works for the
local version), local Auslander regular algebras are Artin-Schelter
regular. So $H^*$ is Artin-Schelter regular of global
dimension $n$. By \cite[Proposition 2.3]{HT}, $H$ satisfies the
Artin-Schelter condition \eqref{E4.1.1}, as required.
\end{proof}

The Calabi-Yau property of a coalgebra is defined and studied in
\cite{HT}, with the Nakayama automorphism of a coalgebra being defined
there also. Every connected graded Artin-Schelter regular algebra is
skew CY as an algebra by \cite[Lemma 1.2]{RRZ}. Following \cite[Theorem 3.2]{HT},
every Artin-Schelter regular connected coalgebra is also called
skew Calabi-Yau. In this connection, we propose a conjecture dual to
\cite[Theorem 0.3]{BZ}.

\begin{conjecture}
\label{yycon4.2}
Let $H$ be a connected Hopf algebra of finite GK-dimension, then
the Nakayama automorphism of the coalgebra $H$ is given by
$S^2$.
\end{conjecture}

If $H$ is a Hopf algebra that is Artin-Schelter Gorenstein
as a coalgebra, then we can define the co-integral as follows.

\begin{definition}
\label{yydef4.3}
Let $H$ be a Hopf algebra satisfying the following
Artin-Schelter condition \cite[Definition 2.1]{HT}
for the coalgebra structure of $H$:
there is an integer $n$ such that
\begin{equation}
\label{E4.3.1}\tag{E4.3.1}
{\text{Ext}}^i_{comod-H}(H,k)=\begin{cases} 0 & i\neq n,\\
S& i=n,\end{cases}
\end{equation}
and
\begin{equation}
\label{E4.3.2}\tag{E4.3.2}
{\text{Ext}}^i_{H-comod}(H,k)=\begin{cases} 0 & i\neq n,\\
T& i=n\end{cases}
\end{equation}
where $S$ and $T$ are 1-dimensional $H$-bi-comodules.
Then the {\it right co-integral} of $H$ is defined to be
the $H$-bi-comodule $S$ in \eqref{E4.3.1} and the
{\it left co-integral} of $H$ is defined to be the
$H$-bi-comodule $T$ in \eqref{E4.3.2}. We say the
co-integrals are trivial if $S$ and $T$ are isomorphic
to the trivial bi-comodule $k$.
\end{definition}

In the connected case, $H$ has trivial co-integrals by \eqref{E4.1.1}.
Conjecture \ref{yycon4.2} should be a special case of a more general
result about the Nakayama automorphism of a coalgebra $H$ when $H$
is a Hopf algebra.

\subsection{The theorem of Latysev-Passman for connected Hopf algebras}
\label{yysec4.2}
As discussed in $\S$\ref{yysec2.6}, Latysev \cite{Lat} and  Passman
\cite{Pas} proved that the enveloping algebra $\mathcal {U}(\mathfrak{g})$
of a Lie $k$-algebra $\mathfrak{g}$ satisfies a polynomial identity
if and only if $\mathfrak{g}$ is abelian. In this subsection we prove an
extension of this result to all connected Hopf algebras. Note that
not every connected Hopf algebra is isomorphic as an algebra to
an enveloping algebra, by Theorem \ref{yythm0.5}(2).

Let $\{H_i : i \geq 0\}$ be the coradical filtration of
the connected Hopf algebra $H$, so $H=\bigcup_{i\geq 0} H_i$.
Then $H_1=k\oplus P(H)$ where $P(H)$ denotes the space of
primitive elements of $H$. Note that $P(H)$ is a Lie
algebra with bracket $[x,y] := xy - yx$ for all
$x,y \in P(H),$ and that $U(P(H))$ is a Hopf subalgebra
of $H$ by \cite[Corollary 5.4.7]{Mo}. For the following
lemma, define $\delta(t)=\Delta(t)-1\otimes t-t\otimes 1$
for any $t\in H$.

\begin{lemma}
\label{yylem4.4}
Let $H$ be a connected Hopf algebra and suppose that the Lie algebra
$P(H)$ is abelian. Fix $i \geq 1$.
\begin{enumerate}
\item[(1)]
 For any $x\in P(H)$ and $y\in H_i$, $[x,y]\in H_{i-1}$.
\item[(2)]
Let $x\in P(H)$ and $y\in H_i$. Then,
for any integer $j$ with $j\geq i$, $x^{j} y=y_{x,j} x$ for some $y_{x,j}\in H$.
\item[(3)]
The multiplicatively closed set of monomials
$\langle P(H)\setminus \{0\} \rangle$ is an Ore set of
regular elements in $H$.
\end{enumerate}
\end{lemma}

\begin{proof} (1) We argue by induction on $i$. Since $P(H)$
is abelian, there is nothing to prove for $i=1$.
Suppose now that $i > 1$ and that the result is proved for
$y \in H_{i-1}.$ Let $x \in P(H)$ and $y\in H_i$. By
definition and \cite[Theorems 5.2.2 and 5.4.1(2)]{Mo}, $\delta(y)\in
\sum_{s=1}^{i-1} H_s\otimes H_{i-s}$. Then
$$\delta([x,y])=
[x\otimes 1+1\otimes x, \delta(y)]\in \sum_{s=1}^{i-1}
[x, H_s]\otimes H_{i-s} +H_s\otimes [x, H_{i-s}].$$
By induction, $[x,H_s]\subset H_{s-1}$ for all $s \leq i-1$,
and $[x,H_1]=0$ since $P(H)$ is abelian.
Hence
$$\delta([x,y])\in \sum_{t=1}^{i-1} H_t\otimes H_{i-1-t}.$$
Thus $[x,y]\in H_{i-1}$, proving the induction step.

(2) We again use induction on $i$. Since $P(H)$ is abelian,
the result holds for $i=1$. Assume now that $i>1$ and that
the result is proved for all $x \in P(H)$ and all $y \in H_{i-1}$.
Let $x \in P(H)$ and $y\in H_i$. By part (1),
$xy=yx+z$ where $z\in H_{i-1}$. By induction, for all $j \geq 0$,
$$x^j y=yx^j+\sum_{s=0}^{j-1} x^s z x^{j-1-s}.$$
Since $z\in H_{i-1}$, for each $j \geq i - 1$ there exists
$z_{x,j} \in H$ such that $x^j z=z_{x,j} x$. Then, for all
$j\geq i$, for some $y' \in H$,
$$x^j y=yx^j+\sum_{s=0}^{j-1} x^s z x^{j-1-s}=
x^{j-1} z+ y' x=z_{x,j-1} x+y' x=y_{x,j} x,$$
and the induction step is proved.

(3) It follows routinely from part (2) that
$\langle P(H)\setminus \{0\} \rangle$ is a left
Ore set in $H$; the argument appears as \cite[Lemma 4.1]{KL},
for example. By symmetry, it is also a right Ore set. Since
$H$ is a domain by \cite[Theorem 6.6]{Zh},
$\langle P(H)\setminus \{0\} \rangle$ consists of regular elements.
\end{proof}

The argument in Lemma \ref{yylem4.4} to deduce parts (2) and (3)
from (1) is essentially the one due to Borho and Kraft which is
reproduced as \cite[Theorem 4.9]{KL}. It is given here for the
reader's convenience.

In the setting of Lemma \ref{yylem4.4}, let $Q(H)$ be the
localization of $H$ obtained by inverting
$\langle P(H)\setminus \{0\} \rangle$.  Note that $H$ is PI if and only
if $Q(H)$ is PI, since if $H$ is PI
then $Q(H)$ is a subalgebra of the central simple quotient
division algebra $Q$ of $H$, which satisfies the same
identities as $H$. With minimal extra effort, we can prove a
local version of the result stated as the $\mathrm{(CCA)}$ case of 
Theorem \ref{yythm0.7}
in the Introduction. We say that a ring $R$ is \emph{locally PI}
if every finite set of elements of $R$ is contained in a subring
of $R$ which satisfies a polynomial identity.

\begin{theorem}
\label{yythm4.5}
Let $H$ be a connected Hopf algebra. Then $H$ is locally PI if and only
if it is commutative.
\end{theorem}

\begin{proof}
If $H$ is commutative, then it is trivially PI.

Conversely, assume that $H$ is locally PI. Every finite
set of elements of $H$ is contained in an affine
Hopf subalgebra of $H$, by \cite[Corollary 3.4]{Zh}. Thus,
in proving that $H$ is commutative, we may assume that $H$
is affine, and hence that $H$ satisfies a
polynomial identity. Since $H=\bigcup_i H_i$, it
suffices to show that elements in $H_i$, for all $i$,
are central in $H$.
For this, we use induction on $i$.

Initial step: Since $H_1=k\oplus P(H)$, the subalgebra of $H$
generated by $H_1$ is $\mathcal {U}(P(H))$.
By \cite[Theorem 1.3]{Pas}, $P(H)$ is abelian, so
Lemma \ref{yylem4.4}
applies. Let $x\in H_1 \setminus k$. We claim, for
each $j \geq 1$, that $x$ commutes with
every element in $H_j$. There is nothing to be
proved when $j=1$. Now assume
$j>1$, and that
\begin{equation}
\label{E4.5.1}\tag{E4.5.1}
[x, H_{j-1}] = 0.
\end{equation}
Let $y\in H_j$, and write $z=[x,y]$. By the definition of the
coradical filtration, $\delta(y)\in H_{j-1}\otimes H_{j-1}$.
By the induction
hypothesis \eqref{E4.5.1},
$$\delta(z)=[1\otimes x+x\otimes 1, \delta(y)]=0.$$
That is, $z\in P(H)$, and in particular, $z$ commutes with $x$.
If $z\neq 0$, then the equation $z=[x,y]$ implies that, in
the localization
$Q(H)$ of $H$,
$$ x(z^{-1} y)-(z^{-1}y) x=1.$$
So $Q(H)$ contains a copy of the first Weyl
algebra, which is not PI since we are in characteristic 0.
This yields a contradiction. Therefore $z=0$ and $x$ commutes
with $y$. So the induction step is proved, and with it the
claim that $H_1$ is central in $H$.

Induction step: Now assume that $i > 1$ and that elements in $H_{i-1}$ are
central in $H$, and let $x\in H_i\setminus H_{i-1}$. We claim that
$x$ commutes with every element in $H_j$ for all $j$. Nothing needs
to be proved when $j\leq i-1$. So assume that $j\geq i$ and that
$[x,H_{j-1}] = 0$. Let $y\in H_j$, so
$\delta(y)\in H_{j-1}\otimes H_{j-1}$. Set $z=[x,y]$.
Then $\delta(z)=[1\otimes x+x\otimes 1, \delta(y)]=0$.
So $z\in P(H)$, and in particular, $z$ commutes with $x$.
By using the same argument as in the initial step, we obtain
$z=0$. This proves the induction step, as required.
\end{proof}

\begin{remarks}
\label{yyrem4.6}
(1) The enveloping algebra of a finite dimensional Lie algebra over a field of
positive characteristic is a finite module over its centre, \cite{Za},
so Theorem \ref{yythm4.5} requires the characteristic 0 hypothesis on $k$.

(2) Since every domain of finite Gel'fand-Kirillov dimension is
an Ore domain by the result of Borho-Kraft
\cite[Theorem 4.12]{KL}, a small recasting of the
above argument yields the following dichotomy:
{\it Let $H$ be a connected Hopf algebra with $\GKdim H < \infty$.
Then either
\begin{enumerate}
\item[(i)]
$H$  is commutative, or
\item[(ii)]
its quotient division ring contains a copy of the first Weyl algebra.
\end{enumerate}}
\end{remarks}

The hypothesis of finite Gel'fand-Kirillov dimension in the above reformulation is
only introduced to guarantee the Ore condition, prompting

\begin{question}
\label{yyque4.7} Let $H$ be a connected Hopf
$k$-algebra. Under what circumstances is $H$ an Ore domain?
\end{question}

Note that, by a well-known result due to Jategaonkar,
\cite[Proposition 4.13]{KL}, $H$, being a domain by \cite[Theorem 6.6]{Zh},
satisfies the Ore condition if it does not contain a free
algebra on two generators.

The above reformulation of Theorem \ref{yythm4.5} also
immediately brings to mind variants of the Gel'fand-Kirillov
conjecture, \cite{AOV}. For example:

\begin{question}
\label{yyque4.8}
Let $H$ be a connected Hopf algebra with $\GKdim H < \infty$.
What conditions on $H$ ensure that its quotient division
ring is a Weyl skew field?
\end{question}

\section{Examples}
\label{yysec5}

The main result of this section is Theorem \ref{yythm5.6},
which answers a number of open questions by giving an example
of a connected Hopf algebra $L$ with $\GK L = 5$, which is also
connected graded as a (Hopf) algebra, but which is not isomorphic
to the enveloping algebra of a Lie algebra.

\subsection{Algebra $H$}
\label{yysec5.1}
First, take two copies of the Heisenberg Lie algebra of dimension 3.
Thus, let
$$\mathfrak{h}_{1}=ka\oplus kb\oplus kc$$
be the Lie algebra with $[a,b]=c$ and $c\in Z(\mathfrak{h}_{1})$. Let
$$\mathfrak{h}_{2}=kz\oplus kw\oplus k{d}$$
where $[z,w]={d}$ and ${d}\in Z(\mathfrak{h}_{2})$.
Set $H:=U(\mathfrak{g})$, where
$\mathfrak{g}=\mathfrak{h}_{1}\oplus\mathfrak{h}_{2}$. By the
Poincar{\'e}-Birkhoff-Witt theorem $H$ has basis $\mathcal{B}$
consisting of the ordered monomials in
$$\{a,b,c,z,w,d\},$$
which we call the $\emph{PBW-generators}$ of $H$.
Since $H$ is an enveloping algebra it comes equipped with a natural
cocommutative coproduct. We show next that there also
exists a non-cocommutative Hopf structure on the algebra $H$.

\subsection{Algebra $J$: defining a new coproduct}
\label{yysec5.2}
Let $J$ be a second copy of the algebra $H$. We now define a
noncocommutative bialgebra structure $(J, \Delta, \epsilon).$

\begin{lemma}
\label{yylem5.1}
Retain the above notation. Then there exist algebra homomorphisms
$\Delta:J\rightarrow J\otimes J$  and $\epsilon:J\rightarrow k$ such that
\begin{equation}\label{E5.1.1}\tag{E5.1.1}
\Delta(a)=1\otimes a+a\otimes 1; \quad \Delta(b)=1\otimes b+b\otimes 1;
\quad \Delta(c)=1\otimes c+c\otimes 1;
\end{equation}
\begin{equation}\label{E5.1.3}\tag{E5.1.2}
\Delta(z)=1\otimes z+a\otimes c-c\otimes a+z\otimes 1;
\end{equation}
\begin{equation}\label{E5.1.4}\tag{E5.1.3}
\Delta(w)=1\otimes w+b\otimes c-c\otimes b+w\otimes 1;
\end{equation}
\begin{equation}\label{E5.1.5}\tag{E5.1.4}
\Delta(d)=1\otimes d+c\otimes c^{2}+c^{2}\otimes c+d\otimes 1;
\end{equation}
and
\begin{equation}\label{E5.1.6}\tag{E5.1.5}
\epsilon(a)=\epsilon(b)=\epsilon(c)=\epsilon(z)=\epsilon(w)=
\epsilon(d)=0.
\end{equation}
With these definitions $(J, \Delta,\epsilon)$ is a bialgebra.
\end{lemma}

The proof of the lemma is a straightforward but long
computation. The interested reader can read it in the
Appendix, $\S$\ref{yysec6}.

\begin{proposition}
\label{yypro5.2}
Retain the above notation. Then $J$ is a noncommutative,
noncocommutative connected Hopf algebra with $\GKdim J=6$.
\end{proposition}

\begin{proof}
From $\S$\ref{yysec5.1}, $J$, as an algebra, is $H$, which has a
PBW-basis $\mathcal{B}$. Define the degree of a PBW-monomial
$p=a^{n_{1}}b^{n_{2}}c^{n_{3}}z^{n_{4}}
w^{n_{5}}d^{n_{6}}\in \mathcal{B}$ to be
$$N(p):=\sum_{i=1}^{3}n_{i}+2(n_{4}+n_{5})+3n_{6}.$$
Define $F_{0}=k$, and for each $n\geq 1$ let $F_{n}$
be the vector space spanned by all
monomials in $\mathcal{B}$ of degree at most $n$. In
particular, $z,w\in F_{2}, d\in F_{3},$ and
\begin{equation}\label{E5.2.2}\tag{E5.2.2}
\Delta(z)=1\otimes z+z\otimes 1+(a\otimes c-c\otimes a)\in
F_{0}\otimes F_{2}+ F_{1}\otimes F_{1}+F_{0}\otimes F_{2};
\end{equation}
\begin{equation}\label{E5.2.3}\tag{E5.2.3}
\Delta(w)=1\otimes w+w\otimes 1+(b\otimes c-c\otimes  b)\in
F_{0}\otimes F_{2}+ F_{1}\otimes F_{1}+F_{0}\otimes F_{2};
\end{equation}
\begin{equation}\label{E5.2.4}\tag{E5.2.4}
\Delta(d)=1\otimes d+d\otimes 1+c^{2}\otimes c+c\otimes c^{2}\in
\sum_{i=0}^{3}F_{i}\otimes F_{3-i}.
\end{equation}
$\textbf{Claim:}$ $\{F_{n}\}$ is an algebra and a coalgebra
filtration of $J$.

That $\{F_{n}\}$ is an exhaustive vector space filtration
of $J$ is immediate. That it is then an algebra filtration
is clear from the defining relations of $J$. Thus it suffices
to prove $\{F_{n}\}$ is a coalgebra filtration. To this end,
let $p \in \mathcal{B}$ be a PBW monomial with $N(p)=n$. We claim
\begin{equation}
\label{E5.2.5}\tag{E5.2.5}
\Delta(p)\in \sum_{i=0}^{n}F_{i}\otimes F_{n-i}.
\end{equation}
Indeed, noting \eqref{E5.2.2}, \eqref{E5.2.3} and \eqref{E5.2.4},
we have
$$\begin{aligned}
\; \Delta(p)&=\Delta(a)^{n_{1}}\Delta(b)^{n_{2}}\Delta(c)^{n_{3}}\Delta(z)^{n_{4}}
\Delta(w)^{n_{5}}\Delta(d)^{n_{6}}\\
& \subseteq \left(\sum_{j_{1}=0}^{{1}}F_{j_{1}}\otimes
F_{1-j_{1}}\right)^{n_{1}}\ldots \left(\sum_{j_{6}=0}^{1}F_{1}\otimes
F_{1-j_{3}}\right)^{n_{3}}\\
& \qquad \left(\sum_{j_{4}=0}^{2}F_{j_{4}}\otimes F_{2-j_{4}}\right)^{n_{4}}
\left(\sum_{j_{5}=0}^{2}F_{j_{5}}\otimes F_{2-j_{5}}\right)^{n_{5}}
\left(\sum_{j_{6}=0}^{3}F_{j_{6}}\otimes F_{3-j_{6}}\right)^{n_{6}}\\
& \subseteq \left(\sum_{j_{1}=0}^{{n_{1}}}F_{j_{1}}\otimes F_{{n_{1}}-j_{1}}
\right)\ldots \left(\sum_{j_{3}=0}^{n_{3}}F_{j_{3}}\otimes F_{n_{3}-j_{3}}
\right)\\
&
\qquad \left(\sum_{j_{4}=0}^{2n_{4}}F_{j_{4}}\otimes F_{2n_{4}-j_{4}}\right)
\left(\sum_{j_{5}=0}^{2n_{5}}F_{j_{5}}\otimes F_{2n_{5}-j_{5}}\right)
\left(\sum_{j_{6}=0}^{3n_{6}}F_{j_{6}}\otimes F_{3n_{6}-j_{6}}\right)\\
&\subseteq\sum_{i=0}^{n}F_{i}\otimes F_{n-i}.
\end{aligned}
$$
where the last two inclusions follow from the fact that $\{F_{n}\}$ is
an algebra filtration. This proves \eqref{E5.2.5}. By
\cite[Lemma 5.3.4]{Mo}, it follows that
$J_{0}\subset F_{0}=k$, and hence $J$ is a connected bialgebra. By
\cite[Lemma 5.2.10]{Mo}, $J$ is thus a connected Hopf algebra. Moreover,
since $J$ is, as an algebra, the enveloping algebra of a Lie algebra of
dimension 6, $\GKdim J = \GKdim {U}(\mathfrak{g})= 6$
by \cite[Example 6.9]{Mo}.
\end{proof}

\subsection{Definition of $L$}
\label{yysec5.3}

Let $J$ be the Hopf algebra defined in $\S$\ref{yysec5.2}.
Using the defining relations of $J$ and
the definition of its coproduct, a straightforward
calculation shows that $d-\frac{1}{3}c^{3}$ is a primitive
element central element of $J$. Let $I$ be the principal
ideal $(d-\frac{1}{3}c^{3})J$ of $J$, and define $L := J/I$.
Recall that an ideal $P$ in a ring $R$ is called
\emph{completely prime} if $R/P$ is a domain.

\begin{proposition}
\label{yypro5.3}
Retain the above notation. Then $L$ is a connected Hopf
algebra with $\GKdim L=5$, and $I$ is a completely prime
Hopf ideal of $J$.
\end{proposition}

\begin{proof}
From the above comments and since $\epsilon(I)=0$, $I$ is a Hopf
ideal of $J$ and hence $L$ is a
Hopf algebra, and is connected by \cite[Corollary 5.3.5]{Mo}
because it is a factor of $J$, which is connected by
Proposition \ref{yypro5.2}. By \cite[Corollary 6.11]{Zh},
$L$, being connected, is a domain. Hence $I$ is a completely prime
ideal.

It remains to show that $\GK L=5$. It is easy to check
that $I'=\langle d,c \rangle$
is a ideal of $J$ and that the factor ring $L':= J/(d,c)$ is isomorphic
as an algebra to the enveloping algebra of the Lie
algebra $\mathfrak{g}/(kd+kc)$. By \cite[Example 6.9]{KL},
$\GK L'=6-2=4$.
We have natural algebra surjections of domains $J\to L\to L'$, so that
\begin{equation}
\label{E5.3.1}\tag{E5.3.1}
 6 = \GK J\geq \GK L \geq \GK L' = 4.
\end{equation}
Moreover, since $L$ is a proper factor of the domain $J$,
the first inequality above is strict. Similarly, $L'$ is a
proper factor of $L$, since a short exercise with the
PBW monomials $\mathcal{B}$ shows that
$c \notin \langle d-\frac{1}{3}c^{3} \rangle$. Thus the second
inequality in \eqref{E5.3.1} is also strict. As
$\GK L \in \mathbb{Z}$ by \cite[Theorem 6.9]{Zh},
it follows that $\GK L = 5.$
\end{proof}

Of course, the fact that $I$ is completely prime can easily be
proved by a direct ring theoretic argument, or, alternatively,
one could use the fact that $\mathfrak{h}_1 \oplus \mathfrak{h}_2$
is a nilpotent Lie algebra, and apply \cite[Theorem 14.2.11]{MR}.

Since $L' = L/cL$ and $L'$ is an enveloping algebra of a
Lie algebra with basis the images of $a,b,z,w$, it is
easy to deduce that $L$ also has a PBW-basis, namely the
ordered monomials in $\{a,b,c, z,w\}.$ Here and below, we
are abusing notation by omitting ``bars" above these elements

For the reader's convenience, we give an explicit presentation
for the Hopf algebra $L$ in the remark below. To calculate the
effect of the antipode on its algebra generators,
we use the fact that $(S_{L}\otimes\operatorname{id})\Delta=\epsilon$.

\begin{remark}
\label{yyrem5.4}
Retain the above notation. Then $L$ is the connected Hopf
algebra with algebra generators $a, b, c, z, w$ subject to
the relations
\[
[a, c] = [a,z] = [a,w] = [b, c] = [b, z] = [b, w] = [c, z] =[c, w] = 0.
\]
and
\[
[a, b] =c, \quad [z,w] =\frac{1}{3}c^{3}.
\]
with coproduct $\Delta:L\rightarrow L\otimes L$, counit
$\epsilon:L\rightarrow k$ and antipode $S:L\rightarrow L$
defined on generators as follows:
\[
\Delta(a)=1\otimes a+a\otimes 1, \quad \Delta(b)=1\otimes b+b\otimes 1,
\quad \Delta(c)=1\otimes c+c\otimes 1,
\]
\[
\Delta(z)=1\otimes z+a\otimes c-c\otimes a+z\otimes 1,
\]
\[
\Delta(w)=1\otimes w+b\otimes c-c\otimes b+w\otimes 1,
\]
and
\[
S(x) = -x; \quad \epsilon(x) = 0
\]
for $x\in\{a, b, c, z, w\}$.
\end{remark}

\subsection{Properties of $L$}
\label{yysec5.4}
For the definition and basic properties of the
\emph{signature} of a connected Hopf
algebra of finite Gel'fand-Kirillov dimension,
see \cite[Definition 5.3]{BG2}; the signature
records the degrees of the homogeneous generators
of the commutative polynomial algebra $\gr_c H$, the
associated graded algebra of $H$ with respect to
its coradical filtration. An \emph{iterated Hopf
Ore extension} (\emph{IHOE}) is a Hopf algebra constructed
as a finite ascending sequence of Hopf subalgebras,
each an Ore extension of the preceding one - see \cite[Definition 3.1]{BO}.

\begin{lemma}
\label{yylem5.5}
Keep the above notation.
\begin{enumerate}
\item[(1)]
$L$ is an IHOE.
\item[(2)]
$\GKdim L = 5$.
\item[(3)]
The signature $\sigma(L)$ of $L$ is $(1,1,1,2,2)$.
\item[(4)]
The centre $Z(L)$ of $L$ is $k[c].$
\item[(5)]
$L$ is connected graded as an algebra and is a graded Hopf
algebra with respect to the same grading.
\end{enumerate}
\end{lemma}

\begin{proof} (1) Adjoin the generators of $L$ in the order
$c,a,b,z,w$ to produce an iterated Ore extension
$$L(0)=k, \quad L(1)=L(0)[c], \quad L(2)=L(1)[a],$$
$$ L(3)=L(2)[b; -c \frac{\partial\;\;}{\partial a}],
\quad
L(4)=L(3)[z], \quad
L(5)=L(4)[w; -\frac{1}{3} c^3 \frac{\partial\;\;}{\partial z}].$$
It is clear from the definition of the coproduct in
$\S$\ref{yysec5.2} that each step in this construction
gives a Hopf subalgebra of $L$; that is, it yields an
iterated Hopf Ore extension in the sense of \cite[Definition 3.1]{BO}.

(2) This follows from part (1) and \cite[Theorem 2.6]{BO}.

(3) Since $a,b,c$ are primitive, the signature $\sigma(L)$ contains
$(1,1,1)$. By the definitions of $\Delta (z)$ and $\Delta (w)$,
the images of the elements $z$ and $w$ are
linearly independent in $P_2(L)/(P(L))$ (see the definitions given
in \cite[Definition 2.4]{Wa1}). Thus, by \cite[Lemma 2.6(e)]{Wa1},
the images of $z$ and $w$ are linearly independent in $P_2(\gr L)/(P(\gr L))$.
Hence the signature $\sigma(L)$ contains $(2,2)$. Since $\GKdim L=5$,
the assertion follows.

(4) Clearly, $k[c] \subseteq Z(L)$. For the reverse inclusion, from
part (1), we can write $L$ as a free (left, say) $k[c]$-module,
with basis the ordered monomials in $a,b,z,w$. By definition,
$$
[a,-]=c\frac{\partial\;\;}{\partial b},\quad
[b,-]=-c\frac{\partial\;\;}{\partial a}, \quad
[z,-]=\frac{1}{3} c^3 \frac{\partial\;\;}{\partial w},\quad
[w,-]=-\frac{1}{3} c^3 \frac{\partial\;\;}{\partial z}.
$$
Each of the inner derivations listed above reduces the total degree of
$a,b,z,w$ by $1$. Suppose that $\alpha \in Z(L) \setminus k[c]$. Obtain a
contradiction by writing $\alpha$ in terms of the above $k[c]$-basis and
applying an appropriate inner derivation from the collection
$[a,-],[b,-],[z,-],[w,-]$ to $\alpha$.

(5) Assign degrees to the generators as follows:
$$ \deg\,a = \deg\,b
= 1; \quad
\deg\,c = 2; \quad
\deg\, z = \deg w = 3.
$$
It is easy to check from the relations that $L$ is then
connected graded as an algebra. Checking the definition of $\Delta$
on the generators, one finds that $L$ is a graded Hopf algebra
with these assigned degrees.
\end{proof}

We can now prove the main result of this section. For
convenience of reference, we restate in it parts of
Lemma \ref{yylem5.5}.

\begin{theorem}
\label{yythm5.6}
Let $L$ be the algebra defined in $\S${\rm{\ref{yysec5.3}}}.
\begin{enumerate}
\item[(1)]
$L$ is a connected Hopf algebra.
\item[(2)]
$L$ is connected graded as an algebra.
\item[(3)]
As an algebra, $L$ is not isomorphic to $U(\mathfrak{n})$ for
any finite dimensional Lie algebra $\mathfrak{n}$.
\item[(4)]
$L$ is an IHOE.
\item[(5)]
$\mathrm{GKdim}L = 5$.
\item[(6)]
The signature $\sigma(L)$ of $L$ is $(1,1,1,2,2)$.
\end{enumerate}
\end{theorem}

\begin{proof} Everything except (3) is proved in Lemma \ref{yylem5.5}.
If $L\cong U(\mathfrak{n})$ for a Lie algebra $\mathfrak{n}$, then
$$\dim_k \mathfrak{n}=\GKdim U(\mathfrak{n})=\GKdim L=5.$$
So, let $\mathfrak{n}$ be a five-dimensional Lie algebra, and
write $U := U(\mathfrak{n}).$ Suppose that, as algebras,
\begin{equation}
\label{E5.6.1}\tag{E5.6.1}
L \cong U
\end{equation}
via the map $\theta: L \rightarrow U$. Easy calculations show that
the commutator ideal $\langle [L,L]\rangle$ of $L$ is the principal ideal
$cL$, and hence that $L/\langle [L,L] \rangle$ is a commutative polynomial
algebra in 4 variables. By \eqref{E5.6.1},
\begin{equation}
\label{E5.6.2}\tag{E5.6.2}
\mathrm{dim}_k(\mathfrak{n}/[\mathfrak{n},\mathfrak{n}]) = 4.
\end{equation}
So $\mathrm{dim}_k([\mathfrak{n},\mathfrak{n}]) = 1$. Write
$[\mathfrak{n}, \mathfrak{n}]=kx$ for some $x\in \mathfrak{n}$. Then the
isomorphism \eqref{E5.6.1} takes $cL$ to $xU$. Since $L$ and $U$ are
domains with only scalars as units, we may assume without loss that
$\theta (c) = x$. In particular, $[\mathfrak{n}, \mathfrak{n}]$ is a
1-dimensional space contained within $Z(\mathfrak{n})$. More precisely,
from Lemma \ref{yylem5.5}(4), $Z(\mathfrak{n}) = kx.$  That is,
\begin{equation}
\label{E5.6.3}\tag{E5.6.3}
\mathfrak{n} \textit{ is nilpotent of class 2, with } Z(\mathfrak{n}) = kx.
\end{equation}

By \cite{Go1} (or use linear algebra), there is up to
isomorphism precisely one Lie algebra of dimension 5 satisfying
\eqref{E5.6.3}. Namely, $\mathfrak{n}$ has a basis
$\{x, x_{1},x_{2}, x_{3}, x_{4} \}$ such that
$$[x_{1}, x_{2}]\quad =\quad x \quad = \quad [x_{3}, x_{4}],$$
and all other brackets are $0$.

Let $L^+$ denote the augmentation ideal of $L$ with respect to its given
Hopf algebra structure. Since the winding automorphisms of $U$ (and
similarly those of $L$) transitively permute the character ideals of $U$
(and, respectively, of $L$), we can follow the map $\theta$ of
\eqref{E5.6.1} by an appropriate winding automorphism of $U$, and thus
assume without loss of generality that
\begin{equation}
\label{E5.6.4}\tag{E5.6.4}
\theta (L^+) = \mathfrak{n}U =
\langle x, x_1, x_2, x_3, x_4 \rangle =: I.
\end{equation}
Define now
\begin{equation}
\label{E5.6.5}\tag{E5.6.5}
B:=L^{+}/(L^{+})^{3} \textit{ and } S := I/I^3.
\end{equation}
By \eqref{E5.6.1} and \eqref{E5.6.4}, $\theta$ induces an isomorphism
(of algebras without identity) of $B$ with $S$. We shall show however that
\begin{equation}
\label{E5.6.6}\tag{E5.6.6}
\mathrm{dim}_k(Z(B)) = 13, \textit{ and } \mathrm{dim}_k(Z(S)) = 11.
\end{equation}
This is manifestly a contradiction, so \eqref{E5.6.1} must be false, and
the proof is complete. It remains therefore to prove \eqref{E5.6.6}.
This is done in Proposition \ref{yypro5.7} below.
\end{proof}

\begin{proposition}
\label{yypro5.7}
Retain the notation introduced in Theorem {\rm{\ref{yythm5.6}}}.
\begin{enumerate}
\item[(1)]
$\mathrm{dim}_k(B) = \mathrm{dim}_k(S) = 15.$
\item[(2)]
$\mathrm{dim}_k(Z(B))= 13.$
\item[(3)]
$\mathrm{dim}_k(Z(S))= 11.$
\end{enumerate}
\end{proposition}

\begin{proof} (1) To prove that
\begin{equation}\label{E5.7.1}\tag{E5.7.1}
\mathrm{dim}_k(B)= 15,
\end{equation}
we show that
\begin{equation}
\label{E5.7.2}\tag{E5.7.2}
B \textit{ has basis (the images of) } \mathcal{B}_B,
\end{equation}
where
\[
\mathcal{B}_B := \{a,b,c,z,w, ab,az,aw,a^{2},b^{2},bz,bw,z^{2},zw,w^{2}\}.
\]
Since (abusing notation regarding images), $L/cL \cong k[a,b,z,w] =: T$,
it follows that
$$\begin{aligned}
T^{+}/(T^{+})^3&=L^+/((L^+)^3 + cL) \\
&= L^+/((L^+)^3 + ck) \\
&= B/\bar{cL},
\end{aligned}
$$
where $\bar{cL}$ is the image of $cL$ in $B$, and
where the second equality holds since $c$ multiplied by any element of
$L^+$ is in $(L^+)^3$. Now $T^+/(T^+)^3$ is clearly a vector space of
dimension 14 with basis the image of $\mathcal{B}_B \setminus \{c\}$.

Moreover,
$$ \bar{cL} = ((L^+)^3 + cL)/(L^+)^3
= ((L^+)^3 + ck)/(L^+)^3,
$$
so that $\mathrm{dim}_k(cB) \leq 1.$ Thus, to prove \eqref{E5.7.1} it
remains to show that
\footnote{It seems this may be more or less contained in
\cite[Corollary 3.3]{RU}, but in any case we sketch here a direct proof.}
\begin{equation}
\label{E5.7.3}\tag{E5.7.3}
c \notin (L^+)^3.
\end{equation}
By definition, $(L^+)^3$ is the $k$-span of the words of length at least
3 in $\{a,b,c,z,w\}$. By a routine straightening argument using the
defining relations for $L$, using for definiteness the ordering
$\{a,b,c,z,w\}$ of the PBW-generators, we calculate that
\begin{equation}
\label{E5.7.4}\tag{E5.7.4}
(L^+)^3 = \mathcal{S} + \mathcal{C},
\end{equation}
where
$$ \mathcal{S} := k-\textit{span of ordered words in } a,b,z,w
\textit{ of length at least }3, $$
and
$$  \mathcal{C} := c \,\times \{k-\textit{span of non-empty
ordered words in } a,b,c,z,w\}. $$
By the linear independence part of the PBW-theorem for $L$, which holds
because $L$ is an IHOE, Lemma \ref{yylem5.5}(1), \eqref{E5.7.3} now
follows from \eqref{E5.7.4}.

The argument for $S$ is similar to the one for $B$. In particular,
one shows that $S$ has basis
$$ \mathcal{B}_S := \{x,x_1,x_2,x_3,x_4\} \cup
\{\textit{ordered words of length 2 in } x_i : 1 \leq i \leq 4 \}.$$

(2) By inspection,
$$ k-\textit{span}\{\mathcal{B}_B \setminus \{a,b\}\}\subset{Z(B)}.$$
Let $f = \lambda a + \mu b + d\in B,$ where
$d \in k-\textit{span}\{\mathcal{B} \setminus \{a,b,\}\}.$ Then, if
$f\in Z(B)$, $0 = [a,f] = \mu c.$ Hence, $\mu = 0.$ Similarly,
$\lambda = 0$, and (2) is proved.

(3) By an argument very similar to (2), one shows that $Z(S)$ has
$k$-basis
$$\mathcal{B}_S \setminus\{x_1,x_2,x_3,x_4\}.$$
This proves (3).
\end{proof}

\begin{remarks}
\label{yyrem5.8}
Let $L$ be the connected Hopf algebra in Theorem \ref{yythm5.6}.
\begin{enumerate}
\item[(1)]
$L$ provides a negative answer to Question L of \cite{BG1}.
\item[(2)]
$L$ is an example of a connected Hopf algebra with
signature $(1,1,1,2,2)$ which is  not isomorphic as a Hopf algebra
to the enveloping algebra of any coassociative Lie algebra.
By contrast, every connected Hopf algebra with signature
$(1,1, \ldots ,1,2)$ \emph{is} isomorphic as an algebra
to such an enveloping algebra. See \cite{Wa3} for the
relevant definition and theorem.
\end{enumerate}
\end{remarks}

\subsection{Proof of Corollary \ref{yycor0.6}}
\label{yysec5.5}
In this subsection we will discuss the dual of $L$ and prove
Corollary \ref{yycor0.6}. We begin with the proof of part (2)
of the corollary, which is immediate from the following lemma,
a variant of a result of Duflo \cite[Theorem 7.2]{KL}.

\begin{lemma}
\label{yylem5.9}
Let $R$ be an augmented $k$-algebra with augmentation
$\epsilon : R \longrightarrow k$ and let $n$ be a
positive integer. Suppose that $R$ is a filtered algebra,
$R = \cup_{i \geq 0} R_i$, with $R_0 = k$ and $\gr R$ a
commutative polynomial algebra in $n$ variables,
generated in degree 1. Then $R \cong U(\mathfrak{g})$,
where $\mathfrak{g}$ is the $n$-dimensional vector space
$R_1/R_0$, which is a Lie algebra with respect to the
bracket induced by the commutator in $R$.
\end{lemma}

\begin{proof}  Since $\gr R = k[x_1, \ldots , x_n]$ is
generated in degree 1, $R$ is also affine, generated
by elements of $R_1$ which are lifts of the $x_i$. In
particular, $R_1$ is finite dimensional, so we can fix
a $k$-basis $y_1, \ldots , y_n$ of
$R_1 \cap \ker \epsilon$ with
$R= k\langle y_1, \ldots , y_n\rangle$. Commutativity
of $\gr R$ and the fact that $\ker \epsilon$
is an ideal of $R$ ensure that $\mathfrak{g} :=
\sum_{i=1}^n ky_i$ is a Lie subalgebra of $R$. The
universal property of the enveloping algebra
$U(\mathfrak{g})$ implies that $R \cong U(\mathfrak{g})/I$
for some ideal $I$ of $U(\mathfrak{g})$. But
$\GKdim R = n = \GKdim U(\mathfrak{g})$ by
\cite[Proposition 6.6, Example 6.9]{KL}, so $I = \{0\}$
by \cite[Proposition 3.15]{KL}.
\end{proof}

We recall some facts about completion in the algebraic setting.
Let $A$ be a $k$-algebra with an ideal $\fm$ with 
$\mathrm{dim}_k A/\fm = 1$. In practice,
$\fm$ is canonically given. Assume that
$\bigcap_{i=0}^{\infty} \fm^i=\{0\}$, so that the
$\fm$-filtration $\{\fm^i\}_{i=0}^{\infty}$
is separated. The {\it completion} of $A$ with respect to the
$\fm$-adic topology  is defined to be
$$\widehat{A}=\lim_{{\longleftarrow}_n} A/\fm^n.$$
The following lemma is well-known.

\begin{lemma}
\label{yylem5.10}
Let $A:=\bigoplus_{i=0}^{\infty} A(i)$ be an affine
connected graded algebra with
$\fm=A_{\geq 1}:=\bigoplus_{i=1}^{\infty} A(i)$. Then
$$\widehat{A}=\lim_{\longleftarrow n} A/A_{\geq n}
=\{\prod_{i=0}^{\infty} a_i\mid a_i\in A(i)\}.$$
As a consequence, a $k$-linear basis of $A$
consisting of homogeneous elements serves as a
topological basis of $\widehat{A}$.
\end{lemma}

\begin{proof}
The second equality is obvious. The first follows from the
fact that $\{A_{\geq n}\}_{n=0}^{\infty}$ is cofinite with
$\{\fm^n \}_{n=0}^{\infty}$. The consequence is clear.
\end{proof}

We also need an easy lemma about the Gel'fand-Kirillov dimension
of a connected coalgebra. Let $C$ be a connected coalgebra
with coradical filtration $\{C_n\}_{n=0}^{\infty}$. Assume that
each $C_n$ is finite dimensional over $k$. Define
the {\it Gel'fand-Kirillov dimension} of $C$ to be
\begin{equation}
\label{E5.10.1}\tag{E5.10.1}
\GKdim_c C:=\overline{\lim_{n\rightarrow \infty}}
\log_n (\dim_k C_n).
\end{equation}

\begin{lemma}
\label{yylem5.11} Let $H$ be a connected Hopf algebra.
Assume that $k$ has characteristic zero as always.
\begin{enumerate}
\item[(1)]
If $H_1$ is infinite dimensional, then $\GKdim H=\infty$.
\item[(2)]
Suppose $H_1$ is finite dimensional. Then
$\GKdim_c H=\GKdim H$.
\end{enumerate}
\end{lemma}

\begin{proof}
(1) In this case, by \cite[Lemma 5.3.3]{Mo} $H$ contains $U(P(H))$ where
$P(H)$ is an infinite dimensional Lie
algebra. The assertion follows from \cite[Lemma 6.5]{KL}.

(2) When $H_1$ is finite dimensional,
each $H_n$ is finite dimensional \cite[Lemma 5.3]{Zh}.
Passing from $H$ to $\gr_c H$ does not change the
$k$-dimension sequence of the coradical filtration, as
noted in \cite[Definition 1.13]{AS1}. Hence
$\GKdim_c H=\GKdim_c \gr_c H$. On the other hand,
by \cite[Theorem 6.9]{Zh},
$\GKdim H=\GKdim \gr_c H$. Hence we can assume that
$H$ is $\gr_c H$ which is both connected graded and coradically
graded \cite[Definition 2.1]{Zh}. By definition, $\GKdim_c H$
is defined by using the Hilbert series of $H$, see \eqref{E5.10.1}.
By Theorem \ref{yythm0.3}(1,5), $\GKdim H$
can be computed by using the Hilbert series of
$H$, in fact, by the same formula as \eqref{E5.10.1}.
Therefore $\GKdim_c H=\GKdim H$.
\end{proof}

Now let $L$ be the graded Hopf algebra defined in $\S$\ref{yysec5.3}.
Let $\fm=\ker \epsilon=L_{\geq 1}$
and let $\widehat{L}$ be the completion of $L$ with respect to the
$\fm$-adic topology. We study $\widehat{L}$ in the next proposition,
part (6) of which is needed to prove Corollary \ref{yycor0.6}(5).
We show in particular that $\widehat{L}$ is an iterated skew power
series algebra of the type whose basic properties are described
in \cite{SV}. In the interest of brevity we leave the routine
checking of some of the details here to the reader.

\begin{proposition}
\label{yypro5.12}
Let $B:=\widehat{L}$ be defined as above.
\begin{enumerate}
\item[(1)]
$\{c^i a^j b^k z^l w^m\mid i,j,k,l,m\geq 0\}$ is
a topological basis of $B$.
\item[(2)]
$B$ is an iterated Ore extension of a similar
pattern to that given in the proof
of Lemma \ref{yylem5.5}, but now in the
setting of formal power series.
\item[(3)]
The centre $Z(B)$ is the formal power series ring $k[[c]]$.
\item[(4)]
The commutator ideal $\langle [B,B]\rangle$ of $B$ is the principal
ideal $cB$.
\item[(5)]
The factor ring $B/\langle [B,B]\rangle$ is a formal power
series ring in four commuting variables.
\item[(6)]
There is no positively graded {\rm{(}}hence nilpotent{\rm{)}}
Lie algebra $\fg$ such that $B$ is isomorphic to
$\widehat{U(\fg)}$, where the completion is with respect to the
maximal ideal $\fb:=\ker \epsilon_{U(\fg)}$.
\end{enumerate}
\end{proposition}

\begin{proof}
(1) This follows from Lemma \ref{yylem5.10}.

(2,3,4,5) These all follow by a straightforward
imitation of the proofs of Lemma \ref{yylem5.5}
and Theorem \ref{yythm5.6}. Note in particular
that the skew formal power series extensions
involve only derivations, and that these
derivations are all \emph{$\mathrm{id}$-nilpotent}
in the sense of \cite[Definition, page 354]{SV},
thus ensuring that the extensions are well-defined.

(6) Note that if $\fg$ is a nilpotent Lie algebra,
then $\bigcap_{i=0}^{\infty} \fb^i=\{0\}$. When $\fg$
is positively graded, the PBW basis of $U(\fg)$ serves
as a topological basis of $\widehat{U(\fg)}$ by
Lemma \ref{yylem5.10}. Now, the key idea of comparison
of finite dimensional factors, used to prove Theorem
\ref{yythm5.6}(3), transfers to this complete setting; details
are omitted. Therefore the assertion holds.
\end{proof}

For the proof of Corollary \ref{yycor0.6}(5) we introduce
some further notation. Let $C$ be a connected coalgebra
with coradical filtration $\{C_n\}$, and let $C^*$ denote the
complete algebra
$$C^* \; := \; \lim_{\longleftarrow n} (C_n)^*.$$
On the other hand, let
$D=\bigoplus_{i=0}^{\infty} D(i)$ be a connected graded coalgebra,
and define a second type of dual algebra, the {\it graded dual} of $D$,
$$D^{\gr *} \; := \; \bigoplus_{i=0}^{\infty} (D(i))^*,$$
so, of course, $D^{\gr *}$ is a connected graded algebra.
Now consider again a connected coalgebra $C$ with coradical
filtration $\{C_n\}$, which is also a connected graded coalgebra,
$$C := \bigoplus_{i=0}^{\infty}C(i), $$
so $C(0) = k = C_0$ and
$\Delta (C(i)) \subseteq \bigoplus_{j=0}^{i}C(j) \otimes C(i-j)$
for $i \geq 0$. Since, by definition,
$C_i = \Delta^{-1}(C \otimes C_{i-1} + C_0 \otimes C)$,
\cite[(5.2.1)]{Mo}, an easy induction on $i$ shows that
$\oplus_{j=0}^i C(j) \subseteq C_i$ for all $i$. Assume
that $C$ is an artinian coalgebra, so in particular
$\dim_k C(i) < \infty$ for all $i$. Then it is
straightforward to show that
\begin{equation}
\label{E5.12.1}\tag{E5.12.1}
C^{*}= \lim_{\longleftarrow n} (C_n)^* \cong
\lim_{\longleftarrow n} (\bigoplus_{i=0}^n C(i))^*
\cong \lim_{\longleftarrow n} C^{\gr *}/(C^{\gr *})_{> n}
=\widehat{C^{\gr *}}.
\end{equation}

Now we are ready to prove Corollary \ref{yycor0.6}.

\begin{proof}[Proof of Corollary \ref{yycor0.6}]
The algebra $L$ defined in $\S$\ref{yysec5.3} is
the algebra $H$ of Corollary
\ref{yycor0.6}.

(1) Since $H$ is connected, this follows from \cite[Theorem 1.2]{Zh}.

(2) If a filtration of $H$ exists such that $\gr H$ is a
polynomial algebra generated in degree 1, then
Lemma \ref{yylem5.9} and Theorem \ref{yythm5.6}(3) yield a contradiction.

(3) This follows from \cite[Proposition 3.4(a)]{GZ}.

(4) This is Lemma \ref{yylem5.5}(3).

(5) Let $D$ be the graded $k$-linear dual of $H$, and
suppose for a contradiction that $K$ is a connected Hopf algebra
such that $D$ is isomorphic, as a coalgebra, to $\gr_c K$.
Then the space of primitive elements of $K$
is, like that of $D$, finite dimensional. By \cite[Theorem 6.9]{Zh}
and Lemma \ref{yylem5.11}(2),
$$\GKdim \gr_c K=\GKdim K=\GKdim_c K=\GKdim_c D=\GKdim D=5.$$
By Theorem \ref{yythm4.1}, $\gr_c K$, $K$, $D$ and $\gr_c D$
are artinian as coalgebras. So we can use \eqref{E5.12.1}.
Let $T$ be the graded dual of $\gr_c K$. Then $T$ is a Hopf
algebra that is connected and cocommutative as a coalgebra and
connected graded, and generated in degree 1, as an algebra. By the
discussion in $\S$0.2 or Theorem \ref{yythm0.3}(2),
$T\cong U(\fg)$ as algebras where $\fg$ is a positively
graded Lie algebra of dimension 5.
Since $D$ is the graded dual of $H$, by \eqref{E5.12.1}
we have algebra isomorphisms,
$$D^*\cong \widehat{D^{\gr *}}
=\widehat{H}.$$
Similarly, by \eqref{E5.12.1},
we have algebra isomorphisms,
$$(\gr_c K)^* \cong \widehat{(\gr_c K)^{\gr *}}
\cong \widehat{T}\cong \widehat{U(\fg)}.$$
Since we assumed that $D\cong \gr_c K$ as coalgebras, we
obtain that $\widehat{H}\cong \widehat{U(\fg)}$,
which contradicts Proposition \ref{yypro5.12}(6).
\end{proof}

\begin{remark}
\label{yyrem5.13}
It is not hard to calculate the Lie algebra
$\mathfrak{g}$ appearing in
Corollary \ref{yycor0.6}(3) - it is the direct
sum of a 3-dimensional Heisenberg algebra with
a 2-dimensional abelian algebra.
\end{remark}

\section{Appendix: Proof of Lemma \ref{yylem5.1}}
\label{yysec6}

Here we give a complete proof of Lemma \ref{yylem5.1}, as follows.

\begin{lemma}
\label{yylem6.1}
Let $J$ be the algebra defined as in $\S$\ref{yysec4.1}.
Let $\Delta:J\rightarrow J\otimes J$  and $\epsilon:J\rightarrow k$ be the
maps such that
\begin{equation}\label{1}
\Delta(a)=1\otimes a+a\otimes 1; \quad \Delta(b)=1\otimes b+b\otimes 1;
\quad \Delta(c)=1\otimes c+c\otimes 1;
\end{equation}
\begin{equation}\label{3}
\Delta(z)=1\otimes z+a\otimes c-c\otimes a+z\otimes 1;
\end{equation}
\begin{equation}\label{4}
\Delta(w)=1\otimes w+b\otimes c-c\otimes b+w\otimes 1;
\end{equation}
\begin{equation}\label{d}
\Delta(d)=1\otimes d+c\otimes c^{2}+c^{2}\otimes c+d\otimes 1;
\end{equation}
and
\begin{equation}\label{5}
\epsilon(a)=\epsilon(b)=\epsilon(c)
=\epsilon(z)=\epsilon(w)=\epsilon(d)=0.
\end{equation}
Then these maps satisfy the following properties:
\begin{enumerate}
\item[(1)]
$\Delta$ defines an algebra homomorphism.
\item[(2)]
$\epsilon$ defines an algebra homomorphism.
\item[(3)]
$\Delta$ is coassociative.
\item[(4)]
$(1\otimes\epsilon)\circ\Delta=1=(\epsilon\otimes 1)\circ\Delta$,
where $1$ denotes the identity map on $J$.
\end{enumerate}
As a result, $(J, \Delta, \epsilon)$ is a bialgebra.
\end{lemma}

\begin{proof} (1) Define a free algebra $F$ on generators
$$\{a,b,c,z,w,d\}$$ and a map $\Delta':F\rightarrow F\otimes F$
defined on these generators as in equations $(\ref{1})-(\ref{d})$.
(Here we are abusing notation by using the same letters for generators of
$F$ and their images in $J$.) Since $F$ is a free algebra on these
generators, we can extend $\Delta$ multiplicatively so that it defines an
algebra homomorphism $F\rightarrow F\otimes F$. Setting $I$ to be the ideal
generated by the relations defining the Lie algebra $\mathfrak{g}$, the
algebra $J$ (or $H$) is the factor algebra $F/I$. Therefore, to check that
there exists an algebra homomorphism $\Delta:J\rightarrow J\otimes J$
defined on generators as in equations $(\ref{1})-(\ref{d})$, it suffices
to prove that the algebra homomorphism $\Delta':F\rightarrow F\otimes F$
defined above satisfies
\begin{equation} \label{E7.1.7}
\Delta'(I)\subseteq I\otimes F+F\otimes I.
\end{equation}
Since $\Delta': F\rightarrow F\otimes F$ is an algebra homomorphism, it
suffices to check this on the generators of $I$, that is, on the relations
defining $\mathfrak{g}$. We list these relations below
$$\begin{aligned}
& [a,b]-c, \;\; [a,c], \;\; [b,c], \;\;   \\
&[a,z], \;\; [b,z],\;\; [c,z],\;\; [a,w],\;\; [b,w],\;\;[c,w],\;\;\\
& [z,w]-d,\;\; [a,d],\;\; [b,d],\;\; [c,d],\;\; [z,d],\;\; [w,d].
\end{aligned}
$$
This is already immediately clear for the ideal generators which involve
only PBW-generators set to be primitive under $\Delta'$, so it suffices
to check only those involving $z,w$ or $d$,. For each $t\in F$, define $\delta'(t)
=\Delta'(t)-t\otimes 1-1\otimes t$. It is trivial that, for $t\in I$,
$\Delta'(t)\in I\otimes F+F\otimes I$ if and only if $\delta'(t)\in
I\otimes F+F\otimes I$. It is easy to check that
$$\delta'([t_1,t_2])=[t_1\otimes 1+1\otimes t_1, \delta'(t_2)]+
[\delta'(t_1), t_2\otimes 1+1+\otimes t_2]+[\delta'(t_1),\delta'(t_2)].$$
Next we check that $\delta'(t)\in
I\otimes F+F\otimes I$ for all $t$ in the third and fourth row of ideal
generators.

\begin{enumerate}
\item{

Generator: $[a,z]$.
$$\begin{aligned}
\; \delta'([a,z])=&[1\otimes a+a\otimes 1, \delta'(z)]=
[1\otimes a+a\otimes 1, a\otimes c-c\otimes a]\\
=&a\otimes [a,c]+[a,c]\otimes a \quad \in \quad I\otimes F+F\otimes I.
\end{aligned}
$$
}

\item{Generator: $[b,z]$.
$$\begin{aligned}
\; \delta'([b,z])=&[1\otimes b+b\otimes 1, \delta'(z)]=
[1\otimes b+b\otimes 1, (a\otimes c-c\otimes a)]\\
=&a\otimes[b,c]-c\otimes [b,a]+[b,a]\otimes c-[b,c]\otimes a\\
=& a\otimes[b,c]-[b,c]\otimes a -c\otimes ([b,a]+c)+([b,a]+c)\otimes c\\
&\in I\otimes F+F\otimes I.
\end{aligned}
$$
}

\item{Generator: $[c,z]$.
$$\begin{aligned}
\; \delta'([c,z])=&[1\otimes c+c\otimes 1, (a\otimes c-c\otimes a)]\\
=&-c\otimes [c,a]+[c,a]\otimes c \quad \in \quad I\otimes F+F\otimes I.
\end{aligned}
$$
}

\item{Generator: $[a,w]$
$$\begin{aligned}
\; \delta'([a,w])=&[1\otimes a+a\otimes 1,(b\otimes c-c\otimes b)]\\
=&b\otimes[a,c]-c\otimes[a,b]+[a,b]\otimes c-[a,c]\otimes b\\
=&b\otimes[a,c]-c\otimes ([a,b]-c)+([a,b]-c)\otimes c-[a,c]\otimes b\\
&\in I\otimes F+F\otimes I.
\end{aligned}
$$
}

\item{Generator: $[b,w]$
$$\begin{aligned}
\;  \delta'([b,w])=&[1\otimes b+b\otimes 1,(b\otimes c-c\otimes b)]\\
&\in I\otimes F+F\otimes I.
\end{aligned}
$$

}

\item{Generator: $[c,w]$
$$\begin{aligned}
\; \delta'([c,w])=&[1\otimes c+c\otimes 1, (b\otimes c-c\otimes b)]\\
&\in I\otimes F+F\otimes I.
\end{aligned}
$$
}

\item{
Generator: $[z,w]-d$.
$$\begin{aligned}
\; \delta'([z,w]-d)=&([1\otimes z+z\otimes 1, (b\otimes c-c\otimes b)]
+[(a\otimes c-c\otimes a), 1\otimes w+w\otimes 1]\\
&+[(a\otimes c-c\otimes a),(b\otimes c-c\otimes b)]-\delta'(d)\\
=&(b\otimes [z,c]-c\otimes [z,b]+[z,b]\otimes c-[z,c]\otimes b)\\
&+(a\otimes [c,w]+[a,w]\otimes c-c\otimes [a,w]-[c,w]\otimes a)\\
&-ac\otimes cb+ca\otimes bc-cb\otimes ac+bc\otimes ac\\
&+([a,b]-c)\otimes c^{2}+c^{2}\otimes ([a,b]-c))\\
=&(b\otimes [z,c]-c\otimes [z,b]+[z,b]\otimes c-[z,c]\otimes b)\\
&+(a\otimes [c,w]+[a,w]\otimes c-c\otimes [a,w]-[c,w]\otimes a)\\
&+[c,a]\otimes cb-ca\otimes [c,b]+[b,c]\otimes ac-bc\otimes [a,c]\\
&+([a,b]-c)\otimes c^{2}+c^{2}\otimes ([a,b]-c))\\
&\in I\otimes F+F\otimes I.
\end{aligned}
$$
}

\item{

Generator: $[a,d]$.
$$\begin{aligned}
\; \delta'([a,d])=&c\otimes [a,c^{2}]+c^{2}\otimes [a,c]+[a,c]\otimes c^{2}
+[a,c^{2}]\otimes c\\
&\in I\otimes F+F\otimes I.
\end{aligned}
$$
}
\item{

Generator: $[b,d]$.
$$\begin{aligned}
\; \delta'([b,d])=&c\otimes [b,c^{2}]+c^{2}\otimes [b,c]+[b,c]\otimes c^{2}
+[b,c^{2}]\otimes c\\
&\in I\otimes F+F\otimes I.
\end{aligned}
$$
}

\item{

Generator: $[c,d]$.
$$\begin{aligned}
\; \delta'([c,d])=&c\otimes [c,c^{2}]+c^{2}\otimes [c,c]+[c,c]\otimes c^{2}
+[c,c^{2}]\otimes c\\
=&0 \quad \in \quad I\otimes F+F\otimes I.
\end{aligned}
$$
}

\item{

Generator: $[z,d]$
$$\begin{aligned}
\; \delta'([z,d])=&[1\otimes z+z\otimes 1, c \otimes c^2 + c^2 \otimes c]
+[b\otimes c - c\otimes b, 1 \otimes d + d \otimes 1]\\
&\quad +[b\otimes c - c\otimes b, c \otimes c^2 + c^2 \otimes c]\\
=&[z,c]\otimes c^2+[z,c^2]\otimes c+c\otimes [z,c^2]+c^2\otimes [z,c]\\
& +b\otimes [c,d]+[b,d]\otimes c -c\otimes [b,d]-[c,d]\otimes b\\
& +[b,c]\otimes c^3+[b,c^2]\otimes c^2-c^2\otimes [b,c^2]-c^3\otimes [b,c]\\
&\in I\otimes F+F\otimes I.
\end{aligned}
$$
}

\item{

Generator: $[w,d]$
$$\begin{aligned}
\; \delta'([w,d])=&
[1\otimes w + w\otimes 1, c \otimes c^2 + c^2 \otimes c]+
[a \otimes c - c\otimes a, 1 \otimes d + d \otimes 1]\\
& \quad [a\otimes c - c\otimes a, c \otimes c^2 + c^2 \otimes c]\\
=&[w,c]\otimes c^2+[w,c^2]\otimes c+c\otimes [w,c^2]+c^2\otimes
[w,c]\\
&+[a,d]\otimes c-c\otimes [a,d]+a\otimes [c,d]-[c,d]\otimes a\\
&+[a,c]\otimes c^3+[a,c^2]\otimes c^2-c^2\otimes [a,c^2]-c^3\otimes [a,c]\\
&\in I\otimes F+F\otimes I.
\end{aligned}
$$
}
\end{enumerate}

Thus there exists an algebra homomorphism $\Delta:J\rightarrow J\otimes J$
defined on the generators $\{a,b,c,z,w, d\}$ as in equations
$(\ref{1})-(\ref{d})$.

(2) Retain the notation of part (1). Define an algebra homomorphism
$\epsilon':F\rightarrow k$ which takes the value $0$ on all generators
$\{a,b,c,z,w, d\}$ of $F$. To check that there exists an algebra
homomorphism $\epsilon:J\rightarrow k$ defined on generators as in
($\ref{5}$), it suffices to check that $\epsilon'(I)= 0$. This is clear.

(3) As noted in \cite[$\S$1]{GZ}, to check that the algebra
homomorphism $\Delta:J\rightarrow J\otimes J$ is coassociative, it
suffices to check this on the algebra generators
$$\{a,b,c,z,w, d\}$$
of $J$. This is already clear for each of the
generators $a,b,c$, so we are left to check for the generators
$z, w$ and $d$. Let $\delta (t)=\Delta(t)-t\otimes 1-1\otimes t$ for
any $t\in J$. It is routine to check that $\Delta$ is coassociative
if and only if $\delta$ is. So it suffices to show that $\delta$
is coassociative on $z$, $w$ and $d$. We calculate

$$\begin{aligned}
(\delta\otimes 1)\circ\delta(z)&=(\delta\otimes 1)(a\otimes c-c\otimes a)\\
&=\delta(a)\otimes c-\delta(c)\otimes a=0,
\end{aligned}
$$
on the other hand,
$$\begin{aligned}
(1\otimes \delta)\circ\delta(z)&=(1\otimes \delta)(a\otimes c-c\otimes a)\\
&=a\otimes \delta(c)-c\otimes \delta(a)=0.
\end{aligned}
$$
So $(\delta\otimes 1)\circ\delta(z)=(\delta\otimes 1)\circ\delta(z)$.
Similarly, $(\delta\otimes 1)\circ\delta(w)=0=(1\otimes \delta)
\circ\delta(w)$.

Note that $\delta(c^2)=2c\otimes c$. Then
$$\begin{aligned}
(\delta\otimes 1)\circ\delta(d)=&
(\delta\otimes 1)(c\otimes c^2+c^2\otimes c)\\
=&\delta(c)\otimes c^2+\delta(c^2)\otimes c\\
=& 2 c\otimes c\otimes c.
\end{aligned}
$$
By symmetry, $(1\otimes \delta)\circ\delta(d)=2c\otimes c\otimes c$.
Therefore $(\delta\otimes 1)\circ\delta(d)=(1\otimes \delta)\circ\delta(d)$.
Therefore $\delta$ (and then $\Delta$) is coassociative.

(4) Again, as noted in \cite[$\S$1]{GZ}, it suffices to check that
$$(\epsilon\otimes 1)\circ\Delta(g)=g=(1\otimes\epsilon)\otimes\Delta(g)$$
for each algebra generator $g$ of $J$. Noting that  $\epsilon(1)=1_{k}$
and that $\epsilon(g)=0$ for all $g\in\{a,b,c,z,w,d\}$, this is clear.
\end{proof}

\end{document}